\documentclass[reqno,a4paper]{amsart}
\usepackage{amscd,amssymb,graphicx,amsfonts,mathrsfs,url,a4wide}
\usepackage[dvipsnames]{xcolor}
\usepackage{float}
\usepackage{graphicx}
\usepackage{psfrag}
\usepackage{bbold}
\usepackage{bm}
\usepackage[colorlinks=true, pdfstartview=FitV, linkcolor=blue,citecolor=blue]{hyperref}
\usepackage{bbm}
\usepackage{tikz}
\usetikzlibrary{shapes,arrows,spy,positioning,decorations,decorations.pathreplacing,scopes,backgrounds,external}
\usepackage{enumitem}
\usepackage{tkz-euclide}
\usepackage{pgfplots}

\usetikzlibrary{external}
\tikzset{external/system call={latex \tikzexternalcheckshellescape -halt-on-error
		-interaction=batchmode -jobname "\image" "\texsource";
		dvips -o "\image".ps "\image".dvi;
		ps2eps "\image.ps"}}

\pgfdeclarelayer{background}
\pgfsetlayers{background,main}
\tikzset{
	my box/.style = {
		, line cap = round
		, line join = round
	}
}

\newcommand{\highlight}[3]{
	\path [my box, line width = #1, draw = #2] #3;
	
	\pgfmathsetmacro{\innerlinewidth}{0.9 * #1}
	\path [my box, line width = \innerlinewidth, draw = #2!20] #3;
}

\usepackage[authoryear]{natbib}
\setcitestyle{citesep={;},open={(},close={)},aysep={}}

\theoremstyle{plain}
\newtheorem{theorem}{Theorem}[section]
\newtheorem{prop}[theorem]{Proposition}
\newtheorem{lemma}[theorem]{Lemma}
\newtheorem{coro}[theorem]{Corollary}

\theoremstyle{definition}
\newtheorem{remark}[theorem]{Remark}

\newtheorem{definition}[theorem]{Definition}

\numberwithin{equation}{section}

\ExplSyntaxOn
\NewDocumentEnvironment{sequation}{O{\small}b}
{
	\yufip_sequation:nnn {equation}{#1}{#2}
}{}
\NewDocumentEnvironment{sequation*}{O{\small}b}
{
	\yufip_sequation:nnn {equation*}{#1}{#2}
}{}
\cs_new_protected:Nn \yufip_sequation:nnn
{
	\begin{#1}
		\mbox{#2\(\displaystyle#3\)}
	\end{#1}
}
\ExplSyntaxOff

\newcommand{\Z}{\mathbb Z}

\newcommand{\N}{\mathbb N}
\renewcommand{\P}{\mathbb P}
\newcommand{\E}{\mathbb E}

\newcommand{\cA}{\mathcal{A}}
\newcommand{\cL}{\mathcal{L}}
\newcommand{\cM}{\mathcal{M}}
\newcommand{\cY}{\mathcal{Y}}

\renewcommand{\epsilon}{\varepsilon}
\renewcommand{\rho}{\varrho}
\newcommand{\ee}{\mathrm{e}}
\newcommand{\dd}{\, \mathrm{d}}

\renewcommand{\geq}{\geqslant}

\renewcommand{\leq}{\leqslant}

\renewcommand{\t}{\vartheta}
\renewcommand{\tt}{\vartheta\nu_1}
\renewcommand{\a}{\alpha}
\renewcommand{\b}{\beta}
\newcommand{\s}{\sigma}

\newcommand{\defeq}{\mathrel{\mathop:}=}
\newcommand{\eqdef}{=\mathrel{\mathop:}} 

\newcommand{\scO}{{\scriptstyle\mathcal{O}}}

\newcommand{\one}{\mathbbm{1}}

\definecolor{darkgreen}{rgb}{0.0, 0.35, 0.13}
\definecolor{lightbrown}{rgb}{0.71, 0.4, 0.11}

\begin{document}
\title{The mutation process on the ancestral line under selection}
\author{E. Baake}
\address[E. Baake]{Faculty of Technology, Bielefeld University, \newline\hspace*{\parindent}Postbox 100131, 33501 Bielefeld, Germany}
\email{ebaake@techfak.uni-bielefeld.de}
\author{F. Cordero}
\address[F. Cordero]{Faculty of Technology, Bielefeld University, \newline\hspace*{\parindent}Postbox 100131, 33501 Bielefeld, Germany}
\email{fcordero@techfak.uni-bielefeld.de}
\author{E. Di Gaspero}
\address[E. Di Gaspero]{Faculty of Technology, Bielefeld University, \newline\hspace*{\parindent}Postbox 100131, 33501 Bielefeld, Germany}
\email{edigaspero@techfak.uni-bielefeld.de}
 
\maketitle

\noindent \emph{keywords:} Moran model, pruned lookdown ancestral selection graph, ancestral line, mutation process, substitution process, phylogeny and population genetics

\bigskip


\section{Abstract}
We consider the Moran model of population genetics with two types, mutation, and selection, and investigate the line of descent of a randomly-sampled individual from a contemporary population. We trace this ancestral line back into the distant past, far beyond the most recent common ancestor of the population (thus connecting population genetics to phylogeny), and analyse the mutation process along this line.

To this end, we use the pruned lookdown ancestral selection graph \citep{len}, which consists of a set of potential ancestors of the sampled individual at any given time. Relative to the neutral case (that is, without selection), we obtain a general bias towards the beneficial type, an increase in the beneficial mutation rate, and a decrease in the deleterious mutation rate. This sheds new light on previous analytical results. We discuss our findings in the light of a
well-known observation at the interface of phylogeny and population genetics, namely, the difference in the mutation rates (or, more precisely, mutation fluxes) estimated via phylogenetic methods relative to those observed in pedigree studies.


\section{Introduction}\label{sec_intro}

\emph{Molecular phylogeny} is concerned with species trees and basically works according to the following principle (for a review of the field, see~\citet[Chs.~14 and 15]{Ewens_Grant_05}). One starts with a collection of contemporary species, samples one individual of each species, and sequences a gene (or a collection of genes) in each of these individuals. From this sequence information, one infers the species tree. This tree lives on the long time scale of (usually many) millions of years.
If one zooms in on the species tree (see Figure~\ref{bigtree_mut2}), one sees population genetics emerge: species (or populations) consist of individuals, and the individual genealogies, on a shorter time scale of, say, a few hundreds of thousands of years, are embedded within the species tree. The connection is made by the ancestral line that starts from the randomly-chosen individual in each species and goes back into the past, beyond the most recent common ancestor of the population, and meets ancestral lines coming from other species.
\begin{figure}
	\begin{center}
		\includegraphics[width=0.9\textwidth]{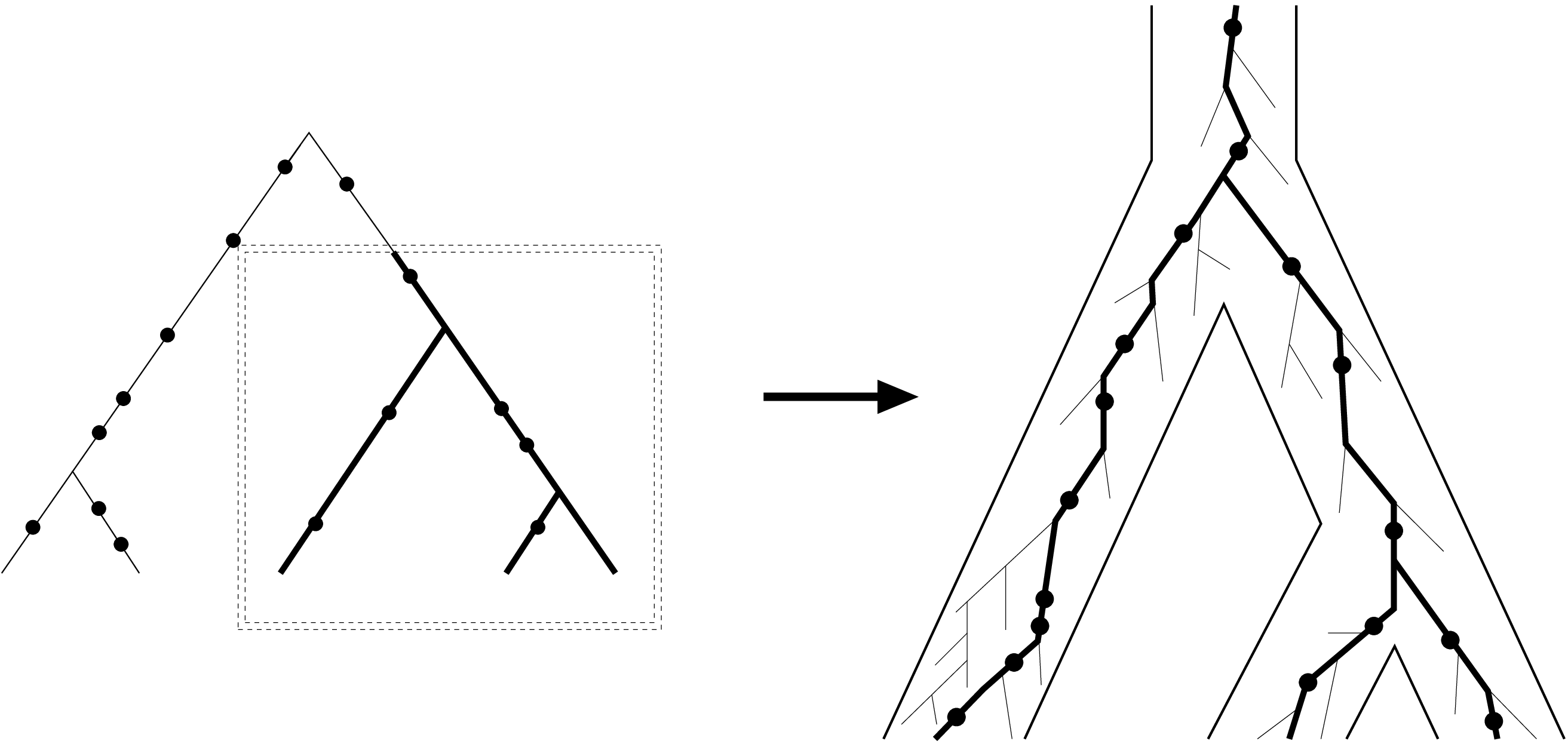}
	\end{center}
	\caption{\label{bigtree_mut2} Phylogeny and population genetics. The bold ancestral line in the right panel makes the connection. The bullets indicate mutation events.}
\end{figure}

\emph{Coalescence theory} is concerned with the individual genealogies within species. It is a major part of modern population genetics theory; for a review of the area, see~\citet{Durrett} or~\citet{Wakeley}. Our goal in this article is to investigate the ancestral line just described, and in this way shed some light on the interface between molecular phylogeny and population genetics. Specifically, we will be concerned with the mutation process on the ancestral line in the presence of selection. Altogether, not much literature is available on the effect of selection in phylogeny; the work of~\citet{Adams_etal_18} and~\citet{Wascher_Kubatko_22} are notable exceptions. We here address the well-known observation that mutation rates (or, more precisely, mutation fluxes) estimated via pedigree studies (that is, via direct comparison of close relatives) in human mitochondrial DNA tend to be about a factor of ten higher than those inferred via phylogenetic methods (\citet{Connell22} and references therein; see specifically the early work by~\citet{Parsons97},~\citet{Sigurdardottir00}, and~\citet{Santos05}). One of the many possible explanations discussed is purifying selection~\citep{Parsons97,Sigurdardottir00,Santos05,Connell22}. This means that a deleterious mutation is selected against; thus an individual carrying this mutation has a lower chance of contributing descendants to future generations than a wildtype individual. So such mutations, although visible in pedigree studies, have a smaller probability of being observed in phylogenetic analyses. Let us note that mutations that contribute to species divergence are often referred to as \emph{substitutions} in the biological literature. 

In this article, we investigate the mutation process on the ancestral line in a prototype model of population genetics, namely the Moran model with two types under mutation and selection. We first describe the model forward in time,
starting with the finite-population version and proceeding to its deterministic and diffusion limits (Section~\ref{sec_model}). In Section~\ref{section_pldASG}, we introduce our main tool, namely the pruned lookdown ancestral selection graph (see~\citet{Baake} for a review), which delivers the type distribution of the ancestor, in the distant past, of an individual sampled randomly from the present population. In Section~\ref{section_backward}, we use this concept to tackle the mutation process on the ancestral line. In the diffusion limit, we recover previous results obtained via analytical methods by~\citet{Fearnhead} and~\citet{Taylor07} and give them a new meaning in terms of the graphical construction. Section~\ref{section_fluxes} analyses in more detail the probability fluxes of mutation and coalescence events involving the ancestral line. Section~\ref{sec_parameters} is devoted to exploring the dependence on the parameters; we discuss the biological implications in Section~\ref{sec:biology}. 

\section{The Model}\label{sec_model}
Consider a population consisting of a fixed number $N\in\N$ of individuals and assume that each of them carries a type, which can be either $0$ or $1$. At any instant in continuous time, every individual may reproduce or mutate. Individuals of type $0$ reproduce at rate $1+s^N$ with $s^N\geq 0$ (where the superscript indicates the dependence on population size), while individuals of type $1$ reproduce at rate $1$. This means that individuals of type $0$ have a selective advantage of $s^N$; for this reason type $0$ is referred to as the fit or beneficial type, whereas type $1$ is the unfit or deleterious type. When an individual reproduces, it produces a single offspring, which inherits the parent's type and replaces an individual chosen uniformly at random from the population (including its parent); in this way, population size is kept constant. Every individual mutates at rate $u^N>0$. When it does so, the new type is $j \in \{0,1\}$ with probability $\nu_j$, where $0<\nu_j<1$ and $\nu_0+\nu_1=1$. This allows for the possibility of silent mutations, where the type is the same before and after the event.

The Moran model can be represented graphically as follows (see Figure~\ref{ips}). Every individual is represented by a horizontal line, lines are labelled $1, \ldots, N$ from top to bottom, time runs from left to right, and the events are indicated by graphical elements: reproduction events are depicted as arrows connecting two lines with the parent at the tail, the offspring at the tip, and the latter replacing the individual at the target line (arrows pointing to their own tails have no effect and are omitted); beneficial and deleterious mutation events are depicted by circles and crosses, respectively, on the lines. We will first describe the \emph{untyped} version, where all graphical elements appear at constant rates, regardless of the types, see Figure~\ref{ips}. We distinguish two kinds of arrows: \emph{neutral} ones have filled arrowheads and appear at rate $1$ per individual (or, equivalently, at rate $1/N$ per ordered pair of lines), while arrows with hollow heads refer to selective reproduction events and appear at rate $s^N/N$ per ordered pair of lines. Finally, circles (crosses) appear at rate $u^N\nu_0$ (at rate $u^N\nu_1$) on every line. 
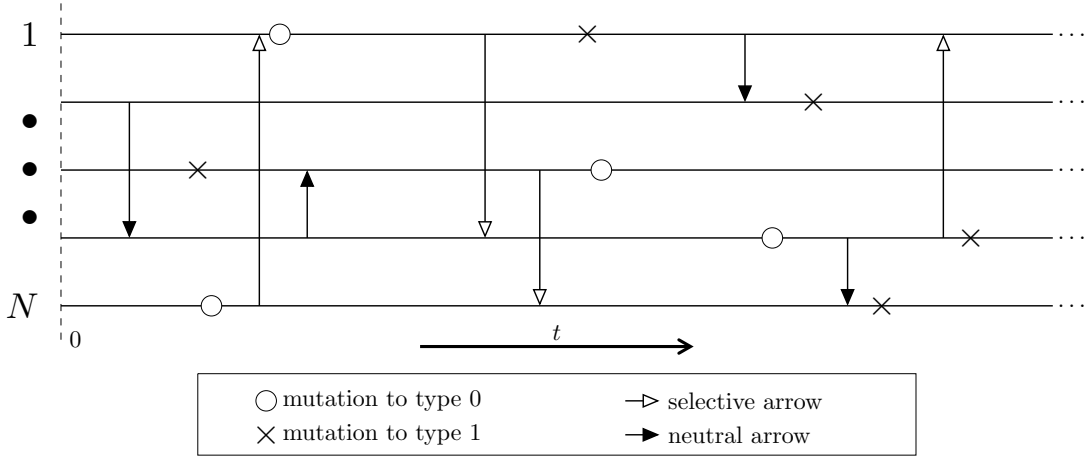
\begin{figure}
	\scalebox{.9}{\begin{tikzpicture}
		
		\draw[dashed] (0,-0.5) --(0,4.5);
		\node [right] at (0,-0.5) {\(0\)};
		\draw[-{angle 60[scale=5]},line width=1.3] (5.25,-0.6) -- (9.25,-0.6) node[text=black, pos=.5, yshift=6pt]{\(t\)};
		
		\draw[opacity=1,semithick] (0,0) -- (14.5,0);
		\draw[opacity=1,semithick] (0,1) -- (14.5,1);
		\draw[opacity=1,semithick] (0,2) -- (14.5,2);
		\draw[opacity=1,semithick] (0,3) -- (14.5,3);
		\draw[opacity=1,semithick] (0,4) -- (14.5,4);
		\draw[-{triangle 45[scale=5]},semithick,opacity=1] (10,4) -- (10,3);
		\draw[-{triangle 45[scale=5]},semithick,opacity=1] (11.5,1) -- (11.5,0);
		\draw[-{triangle 45[scale=5]},semithick,opacity=1] (3.6,1) -- (3.6,2);
		\draw[-{triangle 45[scale=5]},semithick,opacity=1] (1,3) -- (1,1);
		\draw[-{open triangle 45[scale=5]},semithick,opacity=1] (12.9,1) -- (12.9,4);
		\draw[-{open triangle 45[scale=5]},semithick,opacity=1] (2.9,0) -- (2.9,4);
		\draw[-{open triangle 45[scale=5]},semithick,opacity=1] (6.2,4) -- (6.2,1);
		\draw[-{open triangle 45[scale=5]},semithick,opacity=1] (7,2) -- (7,0);
		
		\node[opacity=1] at (11,3) {\scalebox{1.5}{\(\times\)}};
		\node[opacity=1] at (13.3,1) {\scalebox{1.5}{\(\times\)}};
		\node[opacity=1] at (12,0) {\scalebox{1.5}{\(\times\)}};
		\node[opacity=1] at (7.7,4) {\scalebox{1.5}{\(\times\)}};
		\node[opacity=1] at (2,2) {\scalebox{1.5}{\(\times\)}};;  
		
		\draw (2.2,0)[opacity=1] circle (1.5mm)  [fill=white!100];    
		\draw (3.2,4)[opacity=1] circle (1.5mm)  [fill=white!100];
		\draw (10.4,1)[opacity=1] circle (1.5mm)  [fill=white!100];
		\draw (7.9,2)[opacity=1] circle (1.5mm)  [fill=white!100];
		
		\node[opacity=1,left] at (-.2,0) {\scalebox{1.5}{\(N\)}};   
		\node[opacity=1,left] at (-.2,4) {\scalebox{1.5}{\(1\)}};  
		\node[opacity=1,left] at (-.2,1.3) {\scalebox{1.5}{\(\bullet\)}}; 
		\node[opacity=1,left] at (-.2,2) {\scalebox{1.5}{\(\bullet\)}};   
		\node[opacity=1,left] at (-.2,2.7) {\scalebox{1.5}{\(\bullet\)}};
		
		\draw (3,-1.4)  circle (1.5mm)  [fill=white!100] node at (3,-1.4) [right] {\ mutation to type~\(0\)};
		\node at (3,-1.9) {\scalebox{1.5}{\(\times\)}} node at (3,-1.9)[right] {\ mutation to type~\(1\)};    
		\draw[-{open triangle 45[scale=2.5]},semithick] (8.25,-1.4) -- (8.75,-1.4) node [right] {selective arrow};
		\draw[-{triangle 45[scale=2.5]},semithick] (8.25,-1.9) -- (8.75,-1.9) node [right] {neutral  arrow};
		\draw (2,-1) -- (2,-2.2) -- (12.5,-2.2) -- (12.5,-1) -- (2,-1);
		
		\node [right] at (14.35,0) {\(\hspace{0.1cm}\dots\)};
		\node [right] at (14.35,1) {\(\hspace{0.1cm}\dots\)};
		\node [right] at (14.35,2) {\(\hspace{0.1cm}\dots\)};
		\node [right] at (14.35,3) {\(\hspace{0.1cm}\dots\)};
		\node [right] at (14.35,4) {\(\hspace{0.1cm}\dots\)};
\end{tikzpicture}}
	\caption{A realisation of the untyped graphical representation of the Moran model. \label{ips}}
\end{figure}
The untyped version may be turned into a typed one (see Figure \ref{ips_typed}) by assigning an initial type to every line at time $0$ and then propagating the types forward according to the following rules. All individuals use the neutral arrows to place offspring, whereas selective arrows are used only by fit individuals and are ignored by unfit ones. A circle (a cross) turns the type into 0 (into 1). 
\begin{figure}
	\scalebox{.9}{\begin{tikzpicture}

		
		\draw[dashed] (0,-0.5) --(0,4.5);
		\node [right] at (0,-0.5) {\(0\)};
		\draw[-{angle 60[scale=5]},line width=1.3] (5.25,-0.6) -- (9.25,-0.6) node[text=black, pos=.5, yshift=6pt]{\(t\)};

		\draw[-{triangle 45[scale=5]},semithick,opacity=1] (10,4) -- (10,3);
		\draw[-{triangle 45[scale=5]},semithick,opacity=1] (11.5,1) -- (11.5,0);
		\draw[-{triangle 45[scale=5]},semithick,opacity=1] (3.6,1) -- (3.6,2);
		\draw[-{triangle 45[scale=5]},semithick,opacity=1] (1,3) -- (1,1);
		\draw[-{open triangle 45[scale=5]},semithick,opacity=1] (12.9,1) -- (12.9,4);
		\draw[-{open triangle 45[scale=5]},semithick,opacity=1] (2.9,0) -- (2.9,4);
		\draw[-{open triangle 45[scale=5]},semithick,opacity=1] (6.2,4) -- (6.2,1);
		\draw[-{open triangle 45[scale=5]},semithick,opacity=1] (7,2) -- (7,0);
		
		\node[opacity=1,left] at (-.2,0) {\scalebox{1.5}{\(N\)}};   
		\node[opacity=1,left] at (-.2,4) {\scalebox{1.5}{\(1\)}};  
		\node[opacity=1,left] at (-.2,1.3) {\scalebox{1.5}{\(\bullet\)}}; 
		\node[opacity=1,left] at (-.2,2) {\scalebox{1.5}{\(\bullet\)}};   
		\node[opacity=1,left] at (-.2,2.7) {\scalebox{1.5}{\(\bullet\)}};

		\draw[lightbrown,line width=0.9mm] (0.7,-1.4) -- (1,-1.4) node at (1,-1.4)[right,black] {type~{\color{lightbrown}\(1\)} ({\color{lightbrown} unfit})};
		\draw[darkgreen,line width=0.9mm] (0.7,-1.9) -- (1,-1.9) node at (1,-1.9) [right, black] {type~{\color{darkgreen}\(0\)} ({\color{darkgreen}fit})};
		\draw (5.2,-1.4)  circle (1.5mm)  [fill=white!100] node at (5.2,-1.4) [right] {\ mutation to type~\(0\)};
		\node at (5.2,-1.9) {\scalebox{1.5}{\(\times\)}} node at (5.2,-1.9)[right] {\ mutation to type~\(1\)};    
		\draw[-{open triangle 45[scale=2.5]},semithick] (9.7,-1.4) -- (10.2,-1.4) node [right] {selective arrow};
		\draw[-{triangle 45[scale=2.5]},semithick] (9.7,-1.9) -- (10.2,-1.9) node [right] {neutral  arrow};
		\draw (0.5,-1) -- (0.5,-2.2) -- (14,-2.2) -- (14,-1) -- (0.5,-1);

		\draw[darkgreen,opacity=1,line width=0.9mm] (0,0)--(12,0);
		\draw[opacity = 0.7,lightbrown,line width=0.9mm] (12,0)--(14.5,0);
		
		\draw[darkgreen,opacity=1,line width=0.9mm] (0,1)--(1,1);
		\draw[opacity = 0.7,lightbrown,line width=0.9mm] (1,1)--(6.2,1);
		\draw[darkgreen,opacity=1,line width=0.9mm] (6.2,1)--(13.3,1);
		\draw[opacity = 0.7,lightbrown,line width=0.9mm] (13.3,1)--(14.5,1);
		
		\draw[darkgreen,opacity=1,line width=0.9mm] (0,2)--(2,2);
		\draw[opacity = 0.7,lightbrown,line width=0.9mm] (2,2)--(9.4,2);
		\draw[darkgreen,opacity=1,line width=0.9mm] (7.9,2)--(14.5,2);
		
		\draw[opacity = 0.7,lightbrown,line width=0.9mm] (0,3)--(14.5,3);

		\draw[opacity = 0.7,lightbrown,line width=0.9mm] (0,4)--(2.9,4);
		\draw[darkgreen,opacity=1,line width=0.9mm] (2.9,4)--(7.7,4);
		\draw[opacity = 0.7,lightbrown,line width=0.9mm] (7.7,4)--(12.9,4);
		\draw[darkgreen,opacity=1,line width=0.9mm] (12.9,4)--(14.5,4);
		
		\fill [darkgreen] (-.15,-.15) rectangle (0.15,0.15);
		\fill [darkgreen] (-.15,0.85) rectangle (0.15,1.15);
		\fill [darkgreen] (-.15,1.85) rectangle (0.15,2.15);
		\fill [opacity = 0.7,lightbrown] (-.15,2.85) rectangle (0.15,3.15);
		\fill [opacity = 0.7,lightbrown] (-.15,3.85) rectangle (0.15,4.15);
		
		
		\node [right] at (14.35,0) {\(\hspace{0.1cm}\dots\)};
		\node [right] at (14.35,1) {\(\hspace{0.1cm}\dots\)};
		\node [right] at (14.35,2) {\(\hspace{0.1cm}\dots\)};
		\node [right] at (14.35,3) {\(\hspace{0.1cm}\dots\)};
		\node [right] at (14.35,4) {\(\hspace{0.1cm}\dots\)};
		
		\node[opacity=1] at (11,3) {\scalebox{1.5}{\(\times\)}};
		\node[opacity=1] at (13.3,1) {\scalebox{1.5}{\(\times\)}};
		\node[opacity=1] at (12,0) {\scalebox{1.5}{\(\times\)}};
		\node[opacity=1] at (7.7,4) {\scalebox{1.5}{\(\times\)}};
		\node[opacity=1] at (2,2) {\scalebox{1.5}{\(\times\)}};
		
		\draw (2.2,0)[opacity=1,fill=white!100] circle (1.5mm);    
		\draw (3.2,4)[opacity=1,fill=white!100] circle (1.5mm);
		\draw (10.4,1)[opacity=1,fill=white!100] circle (1.5mm);
		\draw (7.9,2)[opacity=1,fill=white!100] circle (1.5mm);
		
\end{tikzpicture}
}
	\caption{A typed version of the realisation in Figure~\ref{ips}.\label{ips_typed}}
\end{figure}
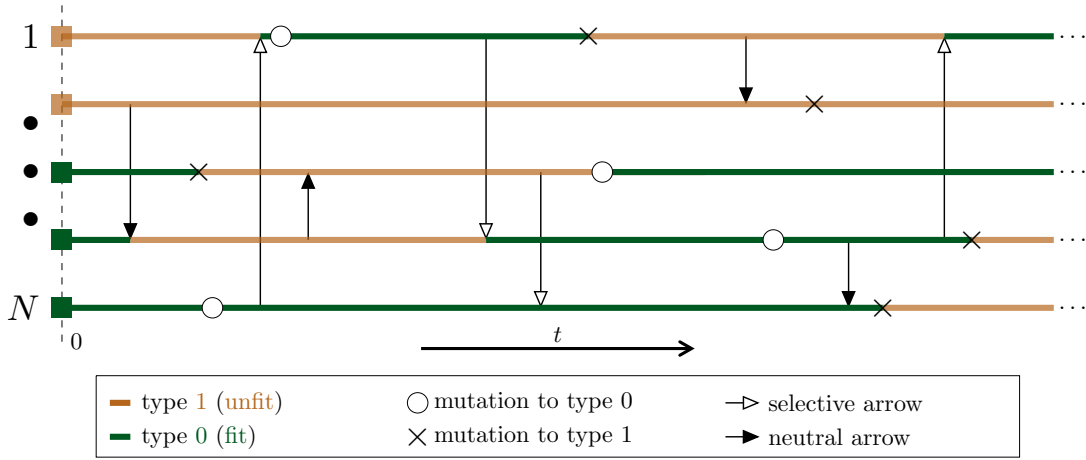
Now let $Y_t^N$ be the number of type-$1$ individuals at time $t$. The process $Y^{N}=(Y_t^{N})_{t\geq 0}$ is a continuous-time birth-death process on $[N]_0$ with transition rates
\begin{equation*}\label{rates_Y}
	\begin{split}
		q^{}_{Y^{N}}(k,k+1) &= k\frac{N-k}{N}+u^N\nu_1 (N-k),\\
		q^{}_{Y^{N}}(k,k-1) &= (1+s^N) (N-k) \frac{k}{N} + u^N\nu_0 k,
	\end{split}
\end{equation*}
where $k\in[N]_0$ (note that $q^{}_{Y^{N}}(N,N+1) = 0 = q^{}_{Y^{N}}(0,-1)$ by construction); throughout, for $i, j \in \N$, we denote by $[i]$ the set $\{1, \ldots, i\}$, by $[i]_0$ the set $\{0, \ldots, i\}$, and by $[i:j]$ the set $\{i, \ldots, j\}$. The (reversible) stationary distribution of $Y^N$, namely $\pi^N = ( \pi^{N}_k)^{}_{k \in [N]_0}$, is given by
\begin{equation}\label{stationarydistributionfinite}
	\pi^{N}_k \defeq C^N \prod_{j=1}^{k-1} \frac{q^{}_{Y^{N}}(j,j+1)}{q^{}_{Y^{N}}(j+1,j)},
\end{equation}
where the empty product is 1 and $C^N$ is a normalising constant chosen so that $\sum_{k=1}^N \pi^{N}_k=1$. We denote by $Y_\infty^{N}$ a random variable with distribution $\pi^{N}$.

One is usually interested in the limiting behaviour as $N\rightarrow\infty$. The most commonly used limits are the deterministic limit (or dynamical law of large numbers) and the diffusion limit, which we will briefly describe below. However, our emphasis will be on finite ~$N$ for the sake of more generality, so as to also allow for scalings intermediate between deterministic and diffusion. For example, the scaling with $N^\alpha$ for $0 < \alpha < 1$ has recently received considerable attention, see, for example, \cite{Boenkost21a,Boenkost21b} or \cite{Esercito}.

\subsection{Diffusion limit}
In the diffusion limit, the selective reproduction rate and the mutation rate depend on $N$ and satisfy
\begin{equation*}
	\lim_{N\to\infty} Ns^N = \sigma, \quad \lim_{N\to\infty} Nu^N =\theta
\end{equation*}
with $0\leq\sigma,\theta<\infty$. Furthermore, it is assumed that $\lim_{N\to\infty} Y_0^N/N = \eta_0$. As $N \to \infty$, the process $(Y_{tN}^N/N)_{t\geq 0}$ is then well known to converge in distribution to $\cY\defeq(\cY_t)_{t\geq 0}$,
the Wright--Fisher diffusion on $[0,1]$ that solves the stochastic differential equation
\begin{equation*}
	\dd \cY_t=\big (-\sigma \cY_t(1-\cY_t)- \theta\nu_0 \cY + \theta\nu_1 (1-\cY_t) \big )\dd t + \sqrt{2 \cY_t (1-\cY_t)}\dd W_t
\end{equation*}
with initial value $\eta_0$,
where $(W_t)_{t \geqslant 0}$ is standard Brownian motion. For $\theta>0$, the process has a stationary measure $\pi$, known as Wright's distribution, with density
\begin{equation}\label{Wright_distrib}
	\pi(\eta) \defeq C\eta^{\theta\nu^{}_1-1}(1-\eta)^{\theta\nu^{}_0-1}\ee^{-\sigma \eta},\hspace{1cm} 0<\eta<1,
\end{equation}
where $C\defeq \int_0^1 \eta^{\theta\nu^{}_1-1}(1-\eta)^{\theta\nu^{}_0-1}\ee^{-\sigma \eta} \dd \eta$ is a normalising constant. Proofs of the above statements can be found in \citet[Chapter 7]{Durrett}. Furthermore, we also have weak convergence of the stationary distribution,
\begin{equation}\label{conv_pi}
	\pi^N \mathrel{\mathop{\longrightarrow}_{N\to\infty}} \pi;
\end{equation} 
indeed, the proof of the corresponding result for the Wright--Fisher model by \citet[Chapter 10, Theorem 2.2]{kurtz} works in exactly the same way for the Moran model. We denote by $\cY_\infty$ a random variable that has the distribution $\pi$.

\subsection{Deterministic limit}
In this case one lets $N \to \infty$ without any rescaling of time or of parameters; so $u^N\equiv u>0$ and $s^N\equiv s$. If $Y_0^N/N \rightarrow y_0 \in [0,1]$ when $N\rightarrow\infty$, then $(Y_t^N/N)_{t\geq 0}$ converges to $(y(t))_{t\geq 0}$, where $(y(t))_{t\geq 0}$ solves the Riccati differential equation
\begin{equation*}
	\dot{y}=-sy(1-y)-u\nu_0 y + u\nu_1(1-y)
\end{equation*}
with initial value $y(0)=y_0$; the convergence is uniform on compact time intervals in probability \cite[Prop. 3.1]{Cordero2}. The solution, which is known explicitly (see for example~\citet{Cordero2}), converges to the unique equilibrium
\begin{equation*}\label{widetilde_y}
	y_\infty \defeq \lim_{t\to\infty} y(t) = 
	\begin{cases}
		\frac{1}{2}\Big(1+\frac{u}{s}-\sqrt{(1-\frac{u}{s})^2+4\frac{u}{s}\nu_0}\Big), & s>0, \\
		\nu_1, & s=0,
	\end{cases}
\end{equation*}
independently of $y^{}_0$. The law of large numbers carries over to the equilibrium in the sense that 
\begin{equation}\label{DistrConvDet}
	\pi^N \mathrel{\mathop{\longrightarrow}_{N\to\infty}} \delta_{y_\infty}
\end{equation}
in distribution~\cite[Corollary 3.6]{Cordero2}, where $ \delta_{y_\infty}$ denotes the point measure on $y_\infty$.

\section{The pruned lookdown ancestral selection graph}\label{section_pldASG}
In this section, we first recall the pruned lookdown ancestral selection graph in the finite case and then take the deterministic and the diffusion limits.

\subsection{The finite case}
Based on the the concept of ancestral selection graph (ASG) as introduced by \citet{ASG} to study ancestries and genealogies of individuals (see also \citet[Sec.~5.4]{Etheridge11}), the pruned lookdown ancestral selection graph (pLD-ASG) is the tool that allows us to study the ancestral line of an individual chosen randomly from a population at any given time $t\geq 0$, to which we refer as the present. It was introduced by \cite{len} for the diffusion limit of the Moran model, extended to the finite population setting and the deterministic limit by \citet[Section 4]{Cordero} and \citet{Baake1}, reviewed by \cite{Baake}, and extended to certain forms of frequency-dependent selection by \cite{Esercito}. The idea behind it is to start with the untyped setting again, follow the history of a single individual (one of the lines at the right of Figure \ref{ips}) backward in time, and keep track of all \emph{potential} ancestors this individual may have. These potential ancestors are \emph{ordered} according to a hierarchy that keeps track of which line will be the \emph{true} ancestor once the types have been assigned to the lines at any given time point in the past. Therefore, at any time $r$ before the present, the pLD-ASG consists of a random number $L_r^N \in [N]$ of lines, the potential ancestors, with labels $1, \ldots, L_r^N$ that reflect the ordering; the upper index indicates the dependence on population size. The ordering contains elements of the lookdown construction of \citet{don} (see also \citet[Sec. 5.3]{Etheridge00}), hence the name; and it is usually visualised by identifying the labels with equidistant levels (starting at level~1), on which the lines are placed. Here we will use a somewhat simpler visualisation in the spirit of the \emph{ordered ASG} of \cite{len}, which also arranges lines from bottom to top according to their labels, but does not insist on equidistant spacing or a start at level~1. But this is only a minor difference in wording and visualisation of an underlying identical concept. 

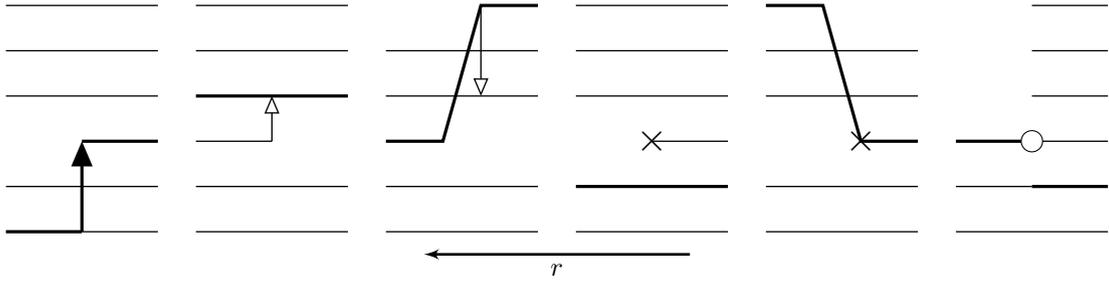
\begin{figure}
	\begin{tikzpicture}[scale=1]
	\draw[latex'-,line width=1] (5.5,-0.3) -- node[anchor=north]{\(r\)} (9,-0.3);
	
	\draw[line width=1.2] (0,0)--(1,0);
	\draw[line width=0.5] (1,0)--(2,0);
	\draw[line width=0.5] (0,0.6)--(2,0.6);
	\draw[line width=1.2] (1,1.2)--(2,1.2);
	\draw[line width=0.5] (0,1.8)--(2,1.8);
	\draw[line width=0.5] (0,2.4)--(2,2.4);
	\draw[line width=0.5] (0,3)--(2,3);
	\draw[-{triangle 45[scale=5]},line width=1.2] (1,0) -- (1,1.2);

	\draw[line width=0.5] (2.5,0)--(4.5,0);
	\draw[line width=0.5] (2.5,0.6)--(4.5,0.6);
	\draw[line width=1.2] (2.5,1.8)--(4.5,1.8);
	\draw[line width=0.5] (2.5,2.4)--(4.5,2.4);
	\draw[line width=0.5] (2.5,3)--(4.5,3);
	\draw[line width=0.5] (2.5,1.2)--(3.5,1.2);
	\draw[-{open triangle 45[scale=5]},line width=0.5] (3.5,1.2) -- (3.5,1.8);

	\draw[line width=0.5] (5,0)--(7,0);
	\draw[line width=0.5] (5,0.6)--(7,0.6);
	\draw[line width=1.2] (5,1.2)--(5.75,1.2)--(6.25,3)--(7,3);
	\draw[line width=0.5] (5,1.8)--(7,1.8);
	\draw[line width=0.5] (5,2.4)--(7,2.4);
	\draw[-{open triangle 45[scale=5]},line width=0.5] (6.25,3) -- (6.25,1.8);

	\draw[line width=0.5] (7.5,0)--(9.5,0);
	\draw[line width=1.2] (7.5,0.6)--(9.5,0.6);
	\draw[line width=0.5] (7.5,1.8)--(9.5,1.8);
	\draw[line width=0.5] (7.5,2.4)--(9.5,2.4);
	\draw[line width=0.5] (7.5,3)--(9.5,3);
	\node[opacity=1] at (8.5,1.2) {\scalebox{1.5}{\(\times\)}};
	\draw[line width=0.5] (8.5,1.2)--(9.5,1.2);
	
	\draw[line width=0.5] (10,0)--(12,0);
	\draw[line width=0.5] (10,0.6)--(12,0.6);
	\draw[line width=0.5] (10,1.8)--(12,1.8);
	\draw[line width=0.5] (10,2.4)--(12,2.4);
	\draw[line width=1.2] (10,3)--(10.75,3)--(11.25,1.2)--(12,1.2);
	\node[opacity=1] at (11.25,1.2) {\scalebox{1.5}{\(\times\)}};
	
	\draw[line width=0.5] (12.5,0)--(14.5,0);
	\draw[line width=0.5] (12.5,0.6)--(13.5,0.6);
	\draw[line width=1.2] (13.5,0.6)--(14.5,0.6);
	\draw[line width=1.2] (12.5,1.2)--(13.5,1.2);
	\draw[line width=0.5] (13.5,1.2)--(14.5,1.2);
	\draw[line width=0.5] (13.5,1.8)--(14.5,1.8);
	\draw[line width=0.5] (13.5,2.4)--(14.5,2.4);
	\draw[line width=0.5] (13.5,3)--(14.5,3);
    \draw (13.5,1.2)[opacity=1,fill=white!100] circle (4pt);

\end{tikzpicture}
	\caption{\label{oASG_finite}Transitions in the finite-$N$ pLD-ASG. The vertical positions reflect the labels, with label 1 at the bottom and label $L_r^N$ at the top. The immune line is bold. From left to right: coalescence, branching, collision, deleterious mutation of a line that is not immune, deleterious mutation of the immune line, beneficial mutation.}
\end{figure}

We throughout denote forward and backward time by $t$ and $r$, respectively; backward time $r$ corresponds to forward time $t-r$. The concept of the pLD-ASG is captured in the following definition, which is adapted from \citet[Sec.~4.6]{Cordero} and \citet[Def.~4]{Baake} and illustrated in Figures~\ref{oASG_finite} and \ref{fig_pldASG_finite}. 

\begin{figure}[H]
	\scalebox{.7}{\begin{tikzpicture}
	\draw[line width = 1.2] (5.2,4.2)--(12,4.2);
	\draw[line width = 1.2] (0.5,2.4)--(1.3,2.4)--(1.5,0)--(3.8,0)--(4,3)--(5.2,3);
	\draw[line width = 1.2] (-1,1.2)--(0.5,1.2);
	\draw(3.3,3.6)--(7.5,3.6);
	\draw[-{open triangle 45},line width=1](7.5,3.6)--(7.5,4.2);
	\draw(5,3)--(10,3);
	\draw[-{open triangle 45},line width=1](10,3)--(10,4.2);
	\draw(2,2.4)--(8,2.4);
	\draw[-{open triangle 45},line width=1](8,2.4)--(8,3);
	\draw(0.5,1.8)--(6.5,1.8);
	\draw[-{open triangle 45},line width=1](6.5,1.8)--(6.5,2.4);
	\draw(0.5,1.2)--(4.8,1.2);
	\draw[-{open triangle 45},line width=1](4.8,1.2)--(4.8,1.8);
	\draw(2.5,0.6)--(9,0.6);
	\draw[-{open triangle 45},line width=1](9,0.6)--(9,3);
	
	\node at (11,4.2) {\scalebox{1.5}{\(\times\)}};
	\node at (2.5,0.6) {\scalebox{1.5}{\(\times\)}};
	\node at (1.5,0) {\scalebox{1.5}{\(\times\)}};
	\node at (6,4.2) {\scalebox{1.5}{\(\times\)}};

	\node[draw,circle,inner sep=3pt,fill=white!100] at (0.5,1.2) {};
	\node[draw,circle,inner sep=3pt,fill=white!100] at (3.3,2.4) {};
	\node[draw,circle,inner sep=3pt,fill=white!100] at (8.5,4.2) {};
	
	\draw[-{triangle 45},line width=1] (5.2,3) -- (5.2,4.2);
	\draw[-{triangle 45},line width=0.5] (2,1.2) -- (2,2.4);
	\draw[-{open triangle 45},line width=1] (4,3) -- (4,0.6);
	
	\begin{pgfonlayer}{background}
		\highlight{4mm}{cyan}{(-1,1.2) -- (4.8,1.2) -- (4.8,1.8) -- (6.5,1.8) -- (6.5,2.4) -- (8,2.4) -- (8,3) -- (10,3) -- (10,4.2) -- (12,4.2)}
	\end{pgfonlayer}
\end{tikzpicture}}
	\caption{A realisation of the finite-$N$ pLD-ASG. The immune line is bold, and the ancestral line is marked light blue. \label{fig_pldASG_finite}}
\end{figure}
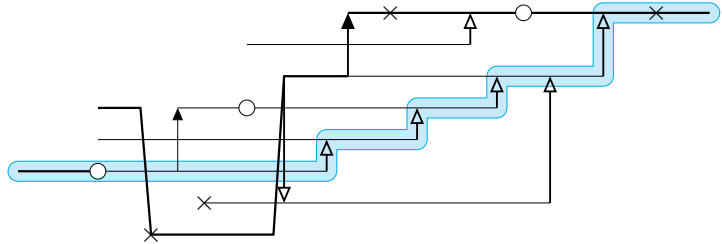

\begin{definition}\label{pld-ASG finite}
	The \emph{pruned lookdown ASG} in a population consisting of $N$ individuals starts with a single line chosen uniformly from the population at backward time $r=0$ and proceeds in direction of increasing $r$. At each time $r\geqslant 0$, the graph consists of a finite number $L^N_r$ of lines, one of which is distinguished and called immune. The lines are labelled $1,\ldots,L^N_r$ from bottom to top; the label of the immune line is $M^N_r \in [L^N_r]$ with $M_0^N=1$. The process $(L^N, M^N) \defeq (L^N_r, M^N_r)^{}_{r \geq 0}$ lives on the state space $[N]^2$ and evolves via the following transitions when $(L^N_r,M^N_r)=(\ell, m)$ (see again Figure~\ref{oASG_finite}).
	\begin{enumerate}
		\item \emph{Coalescence.} Every line $j \in [2: \ell]$ coalesces into every line $i \in [1:j-1]$ at rate $\frac{2}{N}$, that is, label $j$ is removed and $\ell$ decreases by one. The remaining lines are relabelled to fill the gap while retaining their original order. The event affects the immune line in the same way as all others.
		\item \emph{Branching.} Every line $i \in [1:\ell]$ branches at rate $s^N\frac{N-\ell}{N}$, and a new line, namely the incoming branch (the one at the tail of the arrow) is inserted at label $i$ and all previous labels $k\in [i:\ell]$ increase to $k+1$; in particular, the label of the continuing branch (the one at the tip of the arrow) increases from $i$ to $i+1$, and $\ell$ increases to $\ell+1$. If $m\in[i:\ell]$, then $m$ increases to $m+1$; otherwise, it remains unchanged.
		\item \emph{Collision.} Every line labelled $i\in [1:\ell]$ collides with the line labelled $j$ (that is, it performs an \emph{unsuccessful branching event} induced by a selective arrow emerging from line $j$ already in the graph) and is hence relabelled $j$, at rate $s^N$ for every $j \in [1:i-1]$. Upon this, all labels previously in $[j:i-1]$ increase by one. In this case, $\ell$ does not change, and the relabelling applies to the immune line in the same way as to all others.
		\item \emph{Deleterious mutation.} Every line $i\in [1:\ell]$ experiences deleterious mutations at rate $u^N\nu_1$. If $i\neq m$, the line labelled $i$ is pruned, and the remaining lines (including the immune line) are relabelled to fill the gap (again in an order-preserving way), thus rendering the transition of $\ell$ to $\ell-1$. If, however, $i=m$, the line affected by the mutation is not pruned, but $m$ is set to $\ell$, which remains unchanged. All labels previously in $[i+1:\ell]$ decrease by one, so that the gaps are filled.
		\item \emph{Beneficial mutation.} Every line $i\in [1:\ell]$ experiences beneficial mutations at rate $u^N\nu_0$. In this case all lines labelled $>i$ are pruned, resulting in a transition from $\ell$ to $i$. Also $m$ is set to $i$.
	\end{enumerate}
\end{definition}
The line-counting process $L^N \defeq (L^N_r)_{r\geq 0}$, that is, the marginal of $(L^N_r,M^N_r)_{r\geq 0}$, is itself Markov, with state space $[N]$. As a consequence of Definition~\ref{pld-ASG finite}, $L^N$ has transition rates
\begin{equation}\label{pldASG_finite_rates}
	\begin{split}
		q^{}_{L^N}(n,n+1) &= s^Nn\frac{N-n}{N},\\
		q^{}_{L^N}(n,n-1) &= n\frac{n-1}{N}+u^N\nu_1 (n-1)+u^N\nu_0\mathbbm{1}_{\{n>1\}},\\
		q^{}_{L^N}(n,j) &= u^N\nu_0,
	\end{split}
\end{equation}
for $n\in[N]$ and $j\in[n-2]$. 

The pLD-ASG encodes potential and true ancestry in the following way. One first samples the types of the $L^N_t$ lines at backward time $t$ without replacement according to $(N-Y^N_0,Y^N_0)$, which means a population with $N-Y^N_0$ type-0 and $Y^N_0$ type-1 individuals; one then propagates the types forward along the graph according to the coalescence, branching, and mutation events as in the original Moran model. At any instant of time, the true ancestor is the lowest type-$0$ line in the graph if there is such a line; when all individuals are unfit, the ancestor is the immune line \citep[Prop.~4.6]{Cordero}.

Let now $A^N_r\in \{0,1\}$ be the type of the true ancestor at backward time $r \leqslant t$ of the single individual chosen uniformly at $r=0$. It is then clear that
\begin{equation}\label{nonstat}
	\mathbb{P} (A^N_r=1 \mid Y_0^N = k) = \sum_{n=1}^N \P(L_r^N=n) \frac{k^{\underline n}}{N^{\underline n}},
\end{equation}
where, for $k \in \Z$ and $n \in \N$, 
\begin{equation*}
	k^{\underline n} \defeq k(k-1)\ldots(k-n+1)
\end{equation*}
is the falling factorial, and so
\begin{equation}\label{prob_del_anc}
	\P (A^N_r=1) = \sum_{n=1}^N \P(L_r^N=n) \E \Big ( \frac{(Y_0^N)^{\underline n}}{N^{\underline n}} \Big ).
\end{equation}

We are interested in the ancestral type in the distant past, that is, in the limit $r \to \infty$ following $t \to \infty$. 
Thanks to being irreducible and recurrent, $L^N$ converges to a unique stationary distribution as $r \to \infty$; we denote by $L^N_\infty$ a random variable with this stationary distribution and denote the latter by
\begin{equation*}
	w^N \defeq (w^N_n)^{}_{n>0} \quad \text{with } \; w^N_n\defeq \P(L^N_\infty=n). 
\end{equation*}
\citet{Cordero} has shown that the tail probabilities 
\begin{equation*}
	a^N_n \defeq \P(L^N_\infty>n) = \sum_{\ell >n} w^N_\ell
\end{equation*}
are characterised via the recursion 
\begin{equation}\label{rec_a}
	\Big(\frac{n+1}{N}+s \frac{N-n}{N}+u\Big)a_n^N = \Big(\frac{n+1}{N}+u^N \nu_1\Big)a_{n+1}^N + s\frac{N-n}{N}a_{n-1}^N
\end{equation}
for $0<n<N$, together with the boundary conditions $a_0^N = 1$ and $a_N^N=0$.

We additionally assume that, thanks to a long evolutionary prehistory, $Y^N_0$ is stationary (that is, has the distribution $\pi^N$ of \eqref{stationarydistributionfinite});
as a consequence, $Y^N_t$ is stationary for all $t \geqslant 0$.
It is then clear from \eqref{prob_del_anc} that $A^N_r$ will also become stationary as $r \to \infty$. We write $A^N_\infty$ for a random variable with this stationary distribution. Due to \eqref{nonstat} and \eqref{prob_del_anc}, it satisfies
\begin{equation}\label{gamma}
	g^N_k \defeq \P(A^N_\infty=1 \mid Y^N_0 = k) = \sum_{n=1}^N w^N_{n} \frac{k^{\underline n}}{N^{\underline n}}
\end{equation}
and
\begin{equation}\label{anc_unfit}
	\P(A^N_\infty=1) = \sum_{n=1}^N w^{N}_n b^N_n,
\end{equation}
where 
\begin{equation}\label{eq:b^N_n}
	b^N_n \defeq \E\Big ( \frac{(Y^N_0)_{}^{\underline{n}}}{N^{\underline{n}}} \Big ).
\end{equation}
Note that $b^N_n$ is the (unconditional) probability to sample unfit individuals only when drawing $n$ individuals from the stationary population without replacement. It will be shown in Appendix~A (Section~\ref{sec:app_A}) that the $b^N_n$ satisfy the \emph{sampling recursion}
\begin{equation}\label{rec_b}
	\Big(\frac{n-1}{N}+s \frac{N-n}{N}+u\Big)b_n^N = \Big(\frac{n-1}{N}+u^N \nu_1\Big)b_{n-1}^N + s\frac{N-n}{N}b_{n+1}^N, \quad 0 < n \leq N,
\end{equation}
which is complemented by the boundary conditions $b_0^N=1$ and $b_{N+1}^N=0$. Likewise, we can rewrite \eqref{gamma} as
\begin{equation}\label{I_stat_0}
	\begin{split}
		1-g^N_k &= \P (A^N_\infty=0 \mid Y_0^N = k) = 1-\sum_{n=1}^N (a^N_{n-1} - a^N_n) \frac{k^{\underline{n}}}{N^{\underline{n}}}\\
		&= 1 - \frac{k}{N} + \sum_{n=2}^{N} a^N_{n-1} \Big (\frac{k^{\underline{n-1}}}{N^{\underline{n-1}}} - \frac{k^{\underline{n}}}{N^{\underline{n}}} \Big ) = \frac{N-k}{N} + \sum_{n=2}^{N} a^N_{n-1} \frac{k^{\underline{n-1}}}{N^{\underline{n-1}}} (N-k) \\
		&= \sum_{n=1}^N a^N_{n-1} \frac{k^{\underline {n-1}}}{N^{\underline {n}}} (N-k),
	\end{split}
\end{equation}
where we have used in the third step that $a_0^N=1$ and $ a_N^N=0$ and changed summation. The last expression in \eqref{I_stat_0} reflects the decomposition according to the lowest type-0 line. In words: the ancestral line is of type 0 if, for some $n$, there are at least $n$ lines, the first $n-1$ receive type 1, and line $n$ receives type 0. Eq.~\eqref{I_stat_0} also leads to the counterpart of \eqref{anc_unfit},
\begin{equation}\label{anc_fit}
	\begin{split}
		\P(A^N_\infty=0) &= \E \big (\P(A_\infty^N=0 \mid Y_0^N) \big ) \\
		&= \sum_{n=1}^N a^{N}_{n-1} \E\Big ( \frac{(Y^N_0)^{\underline{n-1}}}{N^{\underline{n}}} (N-Y^N_0) \Big ) = \sum_{n=1}^N a^{N}_{n-1} (b^N_{n-1} - b^N_{n}).
	\end{split}
\end{equation}
Also note for later use that, by \eqref{gamma} and \eqref{I_stat_0},
\begin{equation}\label{bad_ind_anc}
	\P(\text{any given type-1 individual is the ancestor} \mid Y^N_0 =k) = \frac{g^N_k}{k} = \sum_{n=1}^N w^N_{n} \frac{(k-1)^{\underline{n-1}}}{N^{\underline n}},
\end{equation}
as well as 
\begin{equation}\label{good_ind_anc}
	\P(\text{any given type-0 individual is the ancestor} \mid Y^N_0 =k) = \frac{1-g^N_k}{N-k} = \sum_{n=1}^N a^N_{n-1} \frac{k^{\underline{n-1}}}{N^{\underline {n}}}; 
\end{equation}
see also \cite{Sandra}, where these relations were derived via a descendant process forward in time.

Let us now compare $\P(A^N_\infty=1)$, the probability for an unfit \emph{ancestor}, with $b^N_1$, the probability for an unfit individual at \emph{present}. Both quantities are presented as functions of $s^N$ in Figure~\ref{fig:anc_dist}. 	To produce this graph (and the others contained in this paper), we have computed the $a^N_n$ and $b^N_n$ by solving numerically the linear systems given by the recursions \eqref{rec_a} and \eqref{rec_b}, together with the approximation $a^N_{n}=b^N_{n}=0$ for $n\geq 100$. This yields accurate results since, for small $s^N$ and large $n$, the $a^N_n$ and $b^N_n$ decrease to zero rapidly, as confirmed by the fact that the results do not change by moving the truncation point to, for example, 50 or 200.

\begin{figure}
	\includegraphics[width=0.6 \textwidth]{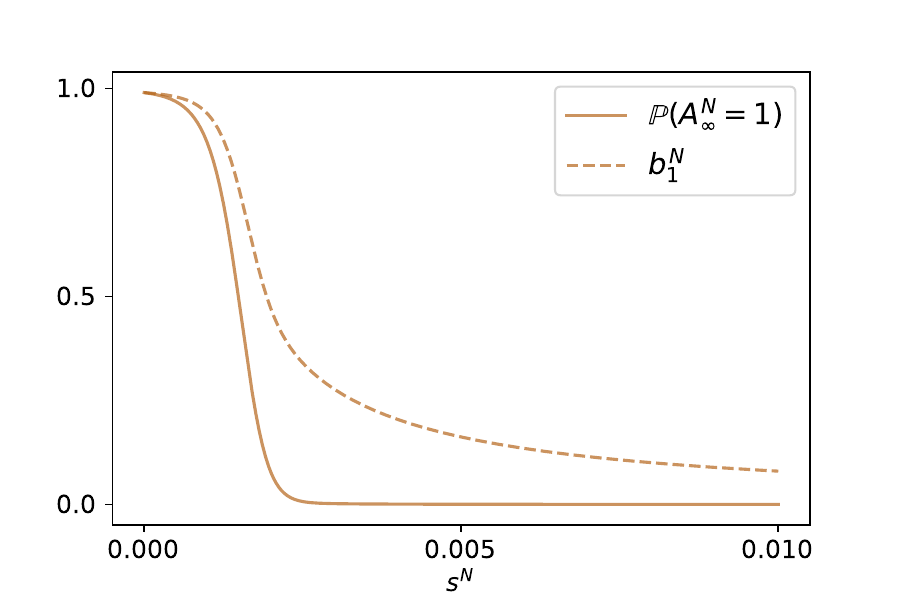}
	\caption{\label{fig:anc_dist} The expected proportion of type-1 individuals in the population (that is, $b_1^N=\E(Y_\infty/N)$) and the probability of a type-1 ancestor in the finite case as a function of $s^N$. Further parameters: $N=10^4, u^N=8 \cdot 10^{-4}, \nu^{}_1=0.99$.}
\end{figure}

It is clear that, for any $s^N>0$, we have $\P(A^N_\infty=1)<b^N_1$, simply because $b^N_n<b^N_1$ for $n>1$, and so $\sum_{n=1}^N w^{N}_n b^N_n < b^N_1$ (compare \eqref{eq:b^N_n} and \eqref{anc_unfit}). This reflects the fact that an unfit individual has a lower chance that its descendants survive in the long run than does a fit individual; this produces the bias towards type 0 in the past relative to the present.

\subsection{Diffusion limit}\label{sec:difflimit}
In the diffusion limit, the pLD-ASG of Definition \ref{pld-ASG finite} changes in the obvious way: (1) coalescence events occur as before (that is, every line coalesces into every line below it at rate 2), (2) branching events happen at rate $\sigma$ to every line, (3) collision events are absent, (4) deleterious mutations and (5) beneficial mutations happen at rate $\theta \nu_1$ and $\theta \nu_0$ per line, respectively, and so do the resulting pruning and relocation operations. In each of these events, the immune line behaves as in the finite case. See \citet{len} and \citet{Baake} for comprehensive descriptions. Here we restrict ourselves to the process $(\cL,\cM) \defeq (\cL_r,\cM_r)_{r \geq 0}$, where $\cL_r$ and $\cM_r$ are the number of lines and the label of the immune line of the limiting pLD-ASG at backward time $r$. $(\cL,\cM)$ lives on the state space $\N^2$, and the marginal $\cL\defeq (\cL_r)_{r \geq 0}$ is again Markov with rates that emerge as the limits of the rates \eqref{pldASG_finite_rates} with the proper rescaling of time and parameters:
\begin{equation*}\label{pldASG_diffusion_rates}
	\begin{split}
		q^{}_{\cL}(n,n+1) &= n\sigma,\\
		q^{}_{\cL}(n,n-1) &= n(n-1)+(n-1)\theta\nu_1 + \theta\nu_0 \mathbbm{1}_{\{n>1\}},\\
		q^{}_{\cL}(n,j) &= \theta\nu_0,
	\end{split}
\end{equation*}
with $j\in[n-2]$, $n\in\N$. The process $\cL$ is non-explosive (because it is dominated by a Yule process with parameter $\sigma$, which is non-explosive), and the sequence of processes $(L^N_{Nr})^{}_{r \geq 0}$ converges to it in distribution as $N \to \infty$; this is a special case of Proposition~2.18 of \citet{Esercito}. The ancestral type $\cA_r \in \{0,1\}$ at backward time $r \leqslant t$ is constructed in analogy with the finite case, except that one works with the pLD-ASG in the diffusion limit and samples the types according to $(1-\cY_\infty,\cY_\infty)$ in an i.i.d.\ manner.

The process $\cL$ again has a unique stationary distribution; we denote by $\cL_\infty$ a random variable with this stationary distribution and let $\omega \defeq (\omega_n)_{n>0}$ with $ \omega_n \defeq \P(\cL_\infty=n) = \lim_{N \to \infty} w_n^N$. It is well known \citep{len} that the tail probabilities $\alpha_n \defeq \P(\cL_\infty>n) = \lim_{N \to \infty} a_n^N$ are the unique solution to the recursion
\begin{equation}\label{recursion}
	(n+1+\sigma+\theta)\alpha_n =(n+1+\theta\nu_1)\alpha_{n+1}+\sigma \alpha_{n-1} \quad \text{for} \; n>0,
\end{equation}
together with $\alpha_0=1$ and $\lim_{n\to\infty} \alpha_n = 0$. Due to \eqref{conv_pi}, we also have that $b^N_n \to \beta_n$ as $N \to \infty$ for $n\geq 0$, where $\beta_n \defeq \E(\cY_\infty^n)$. The $\beta_n$ satisfy the sampling recursion
\begin{equation}\label{rec_beta}
	(n-1+\sigma +\vartheta )\beta^{}_n = \sigma \beta^{}_{n+1} + (n-1+\vartheta \nu^{}_1)\beta^{}_{n-1},\ \quad n\in\N, 
\end{equation} 
complemented by the boundary conditions $\beta_0=1$ and $\lim_{n\to\infty}\beta_n=0$, see \cite{ASG} or \cite{Baake}. 

For $r \to \infty$, the ancestral type becomes stationary again; we denote by $\cA_\infty$ a random variable with this stationary distribution. The stationary probability of a type-1 (a type-0) ancestor is given by the analogue of \eqref{anc_unfit} (of \eqref{anc_fit}), that is,
\begin{equation}\label{anc_unfit_diff}
	\P(\cA_\infty =1) = \sum_{n=1}^\infty \omega^{}_n \beta_n,
\end{equation}
and
\begin{equation*}
	\P(\cA_\infty=0) = \sum_{n=1}^\infty \alpha_{n-1}(\beta_{n-1} - \beta_{n}).
\end{equation*}

\subsection{Deterministic limit}\label{sec:detlimit}
In the deterministic limit, the pLD-ASG of Definition \ref{pld-ASG finite} changes as follows. (1) coalescence events are absent, (2) every line experiences branching events at rate $s$, (3) collisions are absent, (4) deleterious mutations and (5) beneficial mutations happen at rates $u \nu_1$ and $u \nu_0$, respectively, on every line, and so does the resulting pruning (unless the line is immune). In each of these events, the immune line behaves as in the finite case; see \citet{Cordero} or \citet{Baake}.
It is important to note that, in this limit, the immune line is always the top line in the graph, that is at level $L_r$ for all $r \geq 0$. This is because coalescence and collision events are both absent, and thus there is no opportunity to move the immune line away from the top. It therefore suffices to consider the limiting line-counting process $L\defeq (L_r)_{r \geq 0}$, which has state space $\N$ and is characterised by the rates (obtained in the limit from \eqref{pldASG_finite_rates}):
\begin{equation*}\label{pldASG_det_rates}
	\begin{split}
		q^{}_L(n,n+1) &= ns,\\
		q^{}_L(n,n-1) &= (n-1)u\nu_1 + u\nu_0 \mathbbm{1}_{\{n>1\}},\\
		q^{}_L(n,j) &= u\nu_0,
	\end{split}
\end{equation*}
with $j\in [n-2]$, $n\in\N$. The sequence of line-counting processes $(L^N)_{N >0}$ converges in distribution to $L$ as $N \to \infty$ \cite[Prop.~5.3]{Cordero}. The limiting ancestral type $A_r \in \{0,1\}$ at backward time $r \leqslant t$ is constructed in analogy with the finite case, except that one works with the pLD-ASG in the deterministic limit and samples the types according to $(1-y_\infty,y_\infty)$ in an i.i.d.\ manner.

Again, the process $L$ is non-explosive and has a unique stationary distribution \cite[Lemmas~5.1 and 5.2]{Cordero}. We denote by $L_\infty$ a random variable with this distribution, and write ${w \defeq (w_n)_{n >0}}$ with ${w_n \defeq \P(L_\infty=n)= \lim_{N \to \infty} w_n^N}$. The tail probabilities are now given by $a_n\defeq {\P(L_\infty>n)} = \lim_{N \to \infty} a_n^N$. Due to the absence of coalescences and collisions in the pLD-ASG, the distribution is available explicitly, see \citet[Prop.~11]{Baake1} and \citet[Prop.~7]{Baake}:
\begin{equation}\label{geo}
	w_n = p^{n-1} (1-p), \quad a_n = p^n
\end{equation}
(so $L_\infty$ has the geometric distribution Geo$(1-p)$ with parameter $p$), where
\begin{equation}\label{eq_p}
	p \defeq 
	\begin{cases}
		\frac{1}{2}\Big(\frac{u+s}{u\nu_1}-\sqrt{\Big(\frac{u+s}{u\nu_1}\Big)^2-4\frac{s}{u\nu_1}}\Big), & \nu_1>0,\\
		\frac{s}{u+s}, & \nu_1=0.
	\end{cases}
\end{equation}
This also includes the limiting case $p=0$ for $s=0$, where $L_\infty \equiv 1$; the latter is due to the fact that there are neither branching nor coalescence nor pruning events. Furthermore, due to \eqref{DistrConvDet}, we also have $b_n \defeq \lim_{N \to \infty}b_n^N = y^n_\infty$.

In analogy with \eqref{gamma}, we now define 
\begin{equation*}
	g(y) \defeq  \sum_{n=1}^\infty  w_n y^n = (1-p) y \sum_{n=1}^\infty p^{n-1} y^{n-1}. 
\end{equation*}
In analogy with \eqref{I_stat_0}, we then have
\begin{equation*}
	1- g(y) = \sum_{n=1}^\infty a_{n-1} y^{n-1} (1-y) = (1-y) \sum_{n=1}^\infty p^{n-1} y^{n-1}.
\end{equation*}

For $r \to \infty$, the ancestral type becomes stationary again; we denote by $A_\infty$ a random variable with this stationary distribution. The stationary probability of a type-1 (type-0) ancestor is given by the analogue of \eqref{anc_unfit} (of \eqref{anc_fit}), that is,
\begin{equation}\label{anc_unfit_det}
	\P(A_\infty =1) = g(y_\infty)  \quad \text{and} \quad \P(A_\infty=0) = 1-g(y_\infty),
\end{equation}
so
\begin{equation}\label{expoff}
	\frac{\P(A_\infty=1)}{y_\infty}  = \frac{g(y_\infty)}{y_\infty}
	= (1-p) \frac{1-g(y_\infty)}{1-y_\infty} = (1-p) \frac{\P(A_\infty=0)}{1-y_\infty}
\end{equation}
in analogy with \eqref{bad_ind_anc} and \eqref{good_ind_anc}.
The quantities $\P(A_\infty=1)/y_\infty$ and  $\P(A_\infty=0)/(1-y_\infty)$
may  be interpreted as the long-term amount of offspring  of a type-1 and a type-0 individual, respectively, in a stationary population; this will surface again in Remark \ref{georgii}.

\section{The type process on the ancestral line}\label{section_backward}
In this section, we analyse the mutation process of the ancestral line. We again start with the finite case and, from there, proceed to the diffusion and deterministc limits.

\subsection{The finite case}\label{subsec:section_backward_finite}
\begin{figure}[b]
	\input{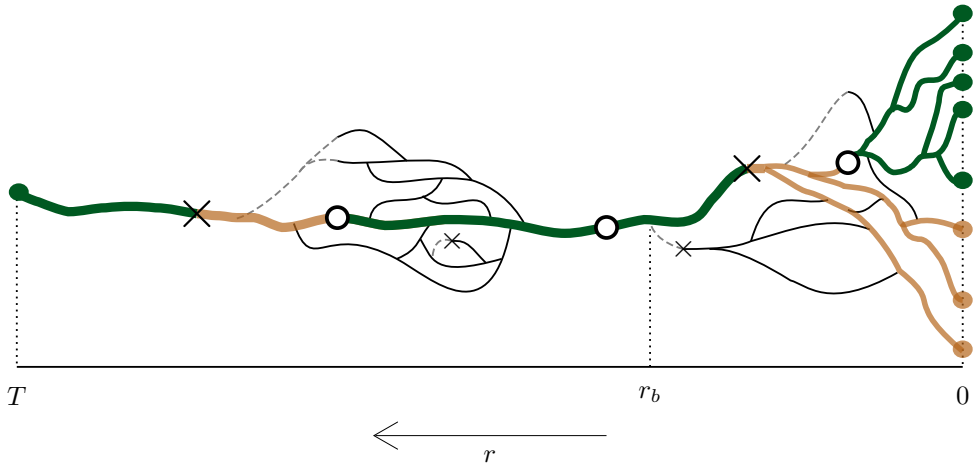}
	\caption{Sketch of a population of size $8$, its (unpruned) ASG decorated with mutations, the common ancestor line (bold), and the common ancestor type process. The first bottleneck is indicated by $r_b$. The coloured lines (dark green for type~0, light brown for type~1) are the true ancestors. The thin black lines are ASG lines that turn out as nonancestral. Among them, the dotted black lines are pruned away in the pLD-ASG. \label{ancline}}
\end{figure}
Consider the ASG decorated with the mutation events but without pruning and start it with all $N$ individuals of the population, rather than with a single individual as assumed so far, see Figure~\ref{ancline}. In analogy with the start with a single line, this captures all potential ancestors of the $N$ individuals. Since the corresponding line-counting process is positive recurrent, it will, with probability $1$, reach a state with just a single line (a \emph{bottleneck}) in finite time. At this point (at the latest), all individuals in the population share a common ancestor. The line of true ancestors leading back into the past from the most recent common ancestor onwards is defined as the \emph{common ancestor line} \citep{Fearnhead, Taylor07, len}. The \emph{common ancestor type process} is the mutation process on this very line, in the forward direction of time. It is often referred to as the \emph{substitution process} in phylogenetics, thus alluding to the fact that these mutations are inherited by all descendants and are thus visible in species comparisons, no matter which specific individuals are sampled.

The ancestral line together with its type can be constructed via a two-step (backward-forward) scheme in any finite time window $[0,T]$, in a way analogous to that described below \eqref{pldASG_finite_rates}. To this end, denote by $t,r\in[0,T]$ the forward and backward running time, respectively (i.e. $r=T-t$). Run the pLD-ASG from time $r=0$ until time $r=T$ started with the $N$ lines, assign a type to each of the ${L}_T^N$ lines by sampling without replacement from $(N-Y_0^N,Y_0^N)$, determine the ancestral line at time $r=T$ (as the lowest type-$0$ line if there is such a line, or the immune line otherwise), and propagate types and ancestry forward as in the original Moran model until time $r=0$. At any time $r$, the true ancestral line is, by construction, the lowest type-$0$ line if there is such a line, or the immune line otherwise. Note that, if there are no bottlenecks in the pLD-ASG in $[0,T]$, the so-constructed ancestral line may not be the ancestor of all the lines at time $r=0$, but for large $T$, the pLD-ASG will experience a bottleneck with high probability, thus ensuring the entire population at time $r=0$ to descend from our ancestral line. 
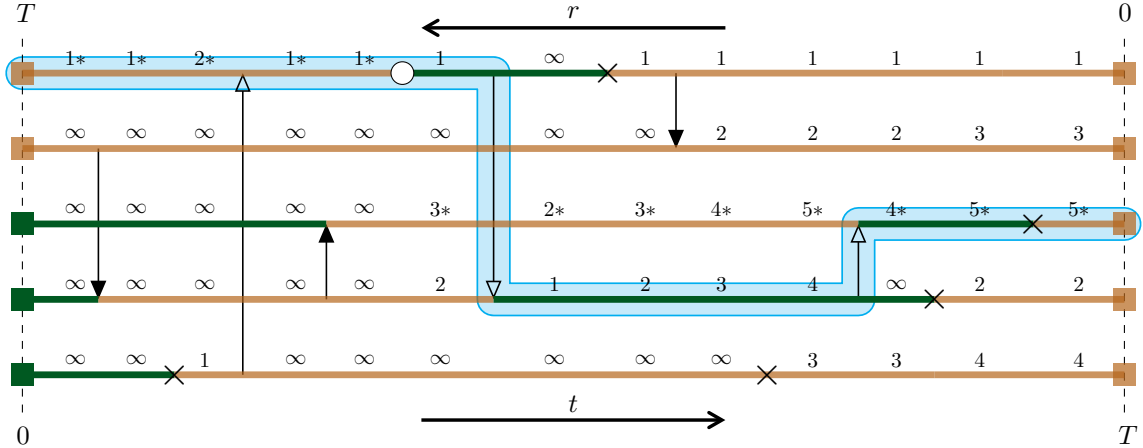
\begin{figure}[t]
	\begin{tikzpicture}
		
		\draw[dashed] (0,-0.5) --(0,4.5);
		\draw[dashed] (14.5,-0.5) --(14.5,4.5);    
		\node [right] at (-0.2,-0.8) {\(0\)};
		\node [right] at (14.3,-0.8) {\(T\)};
		\node [right] at (-0.2,4.8) {\(T\)};
		\node [right] at (14.3,4.8) {\(0\)};
		\draw[-{angle 60[scale=5]},line width=1.3] (5.25,-0.6) -- (9.25,-0.6) node[text=black, pos=.5, yshift=6pt]{\(t\)};
		\draw[-{angle 60[scale=5]},line width=1.3] (9.25,4.6) -- (5.25,4.6) node[text=black, pos=.5, yshift=6pt]{\(r\)};

		\draw[-{triangle 45[scale=5]},semithick,opacity=1] (8.6,4) -- (8.6,3);
		\draw[-{triangle 45[scale=5]},semithick,opacity=1] (4,1) -- (4,2);
		\draw[-{triangle 45[scale=5]},semithick,opacity=1] (1,3) -- (1,1);
		
		\draw[-{open triangle 45[scale=5]},semithick,opacity=1] (11,1) -- (11,2);
		\draw[-{open triangle 45[scale=5]},semithick,opacity=1] (2.9,0) -- (2.9,4);
		\draw[-{open triangle 45[scale=5]},semithick,opacity=1] (6.2,4) -- (6.2,1);

		\draw[darkgreen,opacity=1,line width=0.9mm] (0,0)--(2,0);
		\draw[lightbrown,opacity=0.7,line width=0.9mm] (2,0)--(12,0);
		\draw[lightbrown,opacity=0.7,line width=0.9mm] (12,0)--(14.5,0);
		\node[opacity=1] at (0.7,0.2) {\scalebox{0.8}{\(\infty\)}};
		\node[opacity=1] at (1.5,0.2) {\scalebox{0.8}{\(\infty\)}};
		\node[opacity=1] at (2.4,0.2) {\scalebox{0.8}{\(1\)}};
		\node[opacity=1] at (3.6,0.2) {\scalebox{0.8}{\(\infty\)}};
		\node[opacity=1] at (4.5,0.2) {\scalebox{0.8}{\(\infty\)}};
		\node[opacity=1] at (5.5,0.2) {\scalebox{0.8}{\(\infty\)}};
		\node[opacity=1] at (7,0.2) {\scalebox{0.8}{\(\infty\)}};
		\node[opacity=1] at (8.2,0.2) {\scalebox{0.8}{\(\infty\)}};
		\node[opacity=1] at (9.2,0.2) {\scalebox{0.8}{\(\infty\)}};
		\node[opacity=1] at (10.4,0.2) {\scalebox{0.8}{\(3\)}};
		\node[opacity=1] at (11.5,0.2) {\scalebox{0.8}{\(3\)}};
		\node[opacity=1] at (12.6,0.2) {\scalebox{0.8}{\(4\)}};
		\node[opacity=1] at (13.9,0.2) {\scalebox{0.8}{\(4\)}};
		
		\draw[darkgreen,opacity=1,line width=0.9mm] (0,1)--(1,1);
		\draw[lightbrown,opacity=0.7,line width=0.9mm] (1,1)--(6.2,1);
		\draw[darkgreen,opacity=1,line width=0.9mm] (6.2,1)--(12,1);
		\draw[lightbrown,opacity=0.7,line width=0.9mm] (12,1)--(14.5,1);
		
		\node[opacity=1] at (0.7,1.2) {\scalebox{0.8}{\(\infty\)}};
		\node[opacity=1] at (1.5,1.2) {\scalebox{0.8}{\(\infty\)}};
		\node[opacity=1] at (2.4,1.2) {\scalebox{0.8}{\(\infty\)}};
		\node[opacity=1] at (3.6,1.2) {\scalebox{0.8}{\(\infty\)}};
		\node[opacity=1] at (4.5,1.2) {\scalebox{0.8}{\(\infty\)}};
		\node[opacity=1] at (5.5,1.2) {\scalebox{0.8}{\(2\)}};
		\node[opacity=1] at (7,1.2) {\scalebox{0.8}{\(1\)}};
		\node[opacity=1] at (8.2,1.2) {\scalebox{0.8}{\(2\)}};
		\node[opacity=1] at (9.2,1.2) {\scalebox{0.8}{\(3\)}};
		\node[opacity=1] at (10.4,1.2) {\scalebox{0.8}{\(4\)}};
		\node[opacity=1] at (11.5,1.2) {\scalebox{0.8}{\(\infty\)}};
		\node[opacity=1] at (12.6,1.2) {\scalebox{0.8}{\(2\)}};
		\node[opacity=1] at (13.9,1.2) {\scalebox{0.8}{\(2\)}};
		
		\draw[darkgreen,opacity=1,line width=0.9mm] (0,2)--(4,2);
		\draw[lightbrown,opacity=0.7,line width=0.9mm] (4,2)--(11,2);
		\draw[darkgreen,opacity=1,line width=0.9mm] (11,2)--(13.3,2);
		\draw[lightbrown,opacity=0.7,line width=0.9mm] (13.3,2)--(14.5,2);
		
		\node[opacity=1] at (0.7,2.2) {\scalebox{0.8}{\(\infty\)}};
		\node[opacity=1] at (1.5,2.2) {\scalebox{0.8}{\(\infty\)}};
		\node[opacity=1] at (2.4,2.2) {\scalebox{0.8}{\(\infty\)}};
		\node[opacity=1] at (3.6,2.2) {\scalebox{0.8}{\(\infty\)}};
		\node[opacity=1] at (4.5,2.2) {\scalebox{0.8}{\(\infty\)}};
		\node[opacity=1] at (5.5,2.2) {\scalebox{0.8}{\(3*\)}};
		\node[opacity=1] at (7,2.2) {\scalebox{0.8}{\(2*\)}};
		\node[opacity=1] at (8.2,2.2) {\scalebox{0.8}{\(3*\)}};
		\node[opacity=1] at (9.2,2.2) {\scalebox{0.8}{\(4*\)}};
		\node[opacity=1] at (10.4,2.2) {\scalebox{0.8}{\(5*\)}};
		\node[opacity=1] at (11.5,2.2) {\scalebox{0.8}{\(4*\)}};
		\node[opacity=1] at (12.6,2.2) {\scalebox{0.8}{\(5*\)}};
		\node[opacity=1] at (13.9,2.2) {\scalebox{0.8}{\(5*\)}};
		
		\draw[lightbrown,opacity=0.7,line width=0.9mm] (0,3)--(14.5,3);
	    \node[opacity=1] at (0.7,3.2) {\scalebox{0.8}{\(\infty\)}};
		\node[opacity=1] at (1.5,3.2) {\scalebox{0.8}{\(\infty\)}};
		\node[opacity=1] at (2.4,3.2) {\scalebox{0.8}{\(\infty\)}};
		\node[opacity=1] at (3.6,3.2) {\scalebox{0.8}{\(\infty\)}};
		\node[opacity=1] at (4.5,3.2) {\scalebox{0.8}{\(\infty\)}};
		\node[opacity=1] at (5.5,3.2) {\scalebox{0.8}{\(\infty\)}};
		\node[opacity=1] at (7,3.2) {\scalebox{0.8}{\(\infty\)}};
		\node[opacity=1] at (8.2,3.2) {\scalebox{0.8}{\(\infty\)}};
		\node[opacity=1] at (9.2,3.2) {\scalebox{0.8}{\(2\)}};
		\node[opacity=1] at (10.4,3.2) {\scalebox{0.8}{\(2\)}};
		\node[opacity=1] at (11.5,3.2) {\scalebox{0.8}{\(2\)}};
		\node[opacity=1] at (12.6,3.2) {\scalebox{0.8}{\(3\)}};
		\node[opacity=1] at (13.9,3.2) {\scalebox{0.8}{\(3\)}};

		\draw[lightbrown,opacity=0.7,line width=0.9mm] (0,4)--(5,4);
		\draw[darkgreen,opacity=1,line width=0.9mm] (5,4)--(7.7,4);
		\draw[lightbrown,opacity=0.7,line width=0.9mm] (7.7,4)--(12.9,4);
		\draw[lightbrown,opacity=0.7,line width=0.9mm] (12.9,4)--(14.5,4);
		\node[opacity=1] at (0.7,4.2) {\scalebox{0.8}{\(1*\)}};
		\node[opacity=1] at (1.5,4.2) {\scalebox{0.8}{\(1*\)}};
		\node[opacity=1] at (2.4,4.2) {\scalebox{0.8}{\(2*\)}};
		\node[opacity=1] at (3.6,4.2) {\scalebox{0.8}{\(1*\)}};
		\node[opacity=1] at (4.5,4.2) {\scalebox{0.8}{\(1*\)}};
		\node[opacity=1] at (5.5,4.2) {\scalebox{0.8}{\(1\)}};
		\node[opacity=1] at (7,4.2) {\scalebox{0.8}{\(\infty\)}};
		\node[opacity=1] at (8.2,4.2) {\scalebox{0.8}{\(1\)}};
		\node[opacity=1] at (9.2,4.2) {\scalebox{0.8}{\(1\)}};
		\node[opacity=1] at (10.4,4.2) {\scalebox{0.8}{\(1\)}};
		\node[opacity=1] at (11.5,4.2) {\scalebox{0.8}{\(1\)}};
		\node[opacity=1] at (12.6,4.2) {\scalebox{0.8}{\(1\)}};
		\node[opacity=1] at (13.9,4.2) {\scalebox{0.8}{\(1\)}};
		
		\fill [darkgreen] (-.15,-.15) rectangle (0.15,0.15);
		\fill [darkgreen] (-.15,0.85) rectangle (0.15,1.15);
		\fill [darkgreen] (-.15,1.85) rectangle (0.15,2.15);
		\fill [lightbrown,opacity=0.7] (-.15,2.85) rectangle (0.15,3.15);
		\fill [lightbrown,opacity=0.7] (-.15,3.85) rectangle (0.15,4.15);
		
		\fill [lightbrown,opacity=0.7](14.35,-.15) rectangle (14.65,0.15);
		\fill [lightbrown,opacity=0.7](14.35,0.85) rectangle (14.65,1.15);
		\fill [lightbrown,opacity=0.7](14.35,1.85) rectangle (14.65,2.15);
		\fill [lightbrown,opacity=0.7](14.35,2.85) rectangle (14.65,3.15);
		\fill [lightbrown,opacity=0.7](14.35,3.85) rectangle (14.65,4.15);
		
		\node[opacity=1] at (13.3,2) {\scalebox{1.5}{\(\times\)}};
		\node[opacity=1] at (12,1) {\scalebox{1.5}{\(\times\)}};
		\node[opacity=1] at (9.8,0) {\scalebox{1.5}{\(\times\)}};
		\node[opacity=1] at (7.7,4) {\scalebox{1.5}{\(\times\)}};
		\node[opacity=1] at (2,0) {\scalebox{1.5}{\(\times\)}};
		
		\draw (5,4)[opacity=1] circle (1.5mm)  [fill=white!100];
		
		\begin{pgfonlayer}{background}
			\highlight{4.5mm}{cyan}{(0,4) -- (2.9,4) -- (2.9,4) -- (6.2,4) -- (6.2,1) -- (11,1) -- (11,2) -- (14.5,2)}
		\end{pgfonlayer}
		
\end{tikzpicture}
	\caption{A realisation of the untyped graphical representation of the Moran model in $[0,T]$ for $N=5$ enriched with the (forward) type configuration $(\Xi_t^N)_{t \in [0,T]}$ (given by the colours of the lines: dark green for type 0, light brown for type 1) and the (backward) label processes $(H_r^N)_{r \in [0,T]}$ (given by numbers between $\{1,\ldots,5,\infty\}$ located above each line; the immune line has the additional label $*$). The ancestral line (light blue) is, at any time, the line with the lowest label with type $0$, or the immune line otherwise.\label{ips_tl}}
\end{figure}

Our aim is to construct the common ancestor type process at stationarity in the forward direction of time. Let us anticipate that this process is not Markov; the additional information contained in the graphical representation is required to turn it into a Markov process. As before, let us first consider a finite time window $[0,T]$. We denote by $G_T$ the untyped graphical representation in $[0,T]$ and by $G_T(I)$ its restriction to $I\subseteq[0,T]$. Recall that, given an initial type configuration of the lines present at time $0$ in $G_T$, types can be propagated forward in time along the lines in the graphical representation by respecting the mutation and reproduction events; the so-constructed \emph{type-configuration process} $\Xi^N\defeq(\Xi_t^N)_{t\in[0,T]}$ ($\Xi_t^N \in \{0,1\}^N$ contains the types on all lines at time $t$) is coupled to $G_T$. The pLD-ASG can also be coupled to (embedded in) $G_T$ by assigning, at any time $r$, a label in $\{1,\ldots,N,\infty\}$ to each line in the following way. A label $\infty$ means that the line is currently not in the pLD-ASG, whereas a label $i\in[N]$ indicates the level of that line in the pLD-ASG; an additional label $*$ is used to identify the immune line. So a label is an element of the set \([N]^*_{\infty}\defeq [N]\cup\left\{i* : i\in[N]\right\}\cup\left\{\infty\right\}\). We call the resulting process $H^N\defeq(H_r^N)_{r\in[0,T]}$ the \emph{label process}, and its state space is $([N]^*_{\infty})^N$.

This construction allows us to embed the (forward) type configuration and the (backward) label processes in the same picture (i.e. coupled through $G_T$), see Fig.~\ref{ips_tl}. The process $\mathbb{A}^N\defeq(\mathbb{A}^N_t)_{t\in [0,T]}$, defined via $\mathbb{A}_t^N\defeq (\Xi_t^N,H_{(T-t)-}^N)$, contains all the information required to determine at any time $t\in[0,T]$: (1) the number $Y_t^N$ of individuals of type $1$, (2) the number $\vec{L}_t^N \defeq L^N_{T-t}$ of lines in the time-reversed pLD-ASG, (3) the label $\vec{M}_t^N\defeq M^N_{T-t}$ of the immune line, and (4) the type $A_t^N\in\{0,1\}$ of the ancestral line. Note that the individual processes $\Xi^N$ and $H^N$ can be extended to $[0,\infty)$, but the coupled pair $\mathbb{A}^N$ requires the finite interval $[0,T]$.

By construction, $(H^N)_{r \geq 0}$ is irreducible, and hence it has a unique stationary distribution. Assume now that $\Xi_0^N$ and $H_0^N$ are chosen independently according to their stationary distributions. Since $\Xi_t^N$ and $H_{(T-t)-}^N$ are deterministic functions of their initial configurations and of $G_T([0,t])$ and $G_T((t,T])$, respectively, they are independent (for any fixed time $t$, but not as processes). Thus, the law of $\mathbb{A}_t^N$ is invariant in time and has product form (the product of the corresponding stationary distributions). We will see in Proposition \ref{prop:A_is_Markov} in Appendix~B (Section~\ref{sec:app_B}), that, started with this invariant distribution, the process $\mathbb{A}^N$ has Markovian transitions (and, in particular, we can define it in $[0,\infty)$). Since we are only interested in the transitions of $\mathbb{A}^N$ involving a type change on the ancestral line, we focus on the marginal (non-Markovian) process $A^N$ here and defer the details of the process to Appendix~B.

To get the notions right, let $Z \defeq (Z_t)_{t \geq 0}$ be a (possibly non-Markovian) stochastic process on a finite state space $F$. We understand the probability flux from $z$ to $z' \neq z$, $z, z' \in F$, to be
\begin{equation}\label{eq:marginal_fluxes}
	f_{Z,t}(z,z') \defeq \lim_{\epsilon \to 0} \frac{1}{\epsilon} \P (Z_t=z,Z_{t+\epsilon}=z’)
\end{equation}
if the limit exists; and we take the corresponding rates to be
\begin{equation}\label{eq:marginal_rates}
	q^{}_{Z,t}(z,z') \defeq \frac{f_{Z,t}(z,z')}{\P(Z_t=z)}
\end{equation}
provided $\P(Z_t=z)>0$. If $Z$ is stationary, the fluxes and rates are independent of time, so $ f_{Z,t}(z,z')\equiv f_{Z}(z,z') $ and $q^{}_{Z,t}(z,z') \equiv q^{}_{Z}(z,z')$. Consider now a stationary Markov process $X\defeq(X_t)_{t\geq 0}$ on a finite state space $E$. Let $h:E\to F$ be a surjective bounded function. Suppose the rate of $X$ from $x$ to $x'\neq x$ is $q^{}_{X}(x,x')$, and $\pi^{}_{X}$ is the stationary distribution of $X$. By \eqref{eq:marginal_fluxes} and \eqref{eq:marginal_rates}, the probability fluxes of the mapped process (which may be non-Markovian) always exist and read
\begin{equation}\label{eq:marginal_f}
	\begin{split}
		f_Z(z,z') & = \sum_{\substack{x\in h^{-1}(\{z\}),\\x'\in h^{-1}(\{z'\})}} f_X(x,x') \\
		& = \sum_{\substack{x\in h^{-1}(\{z\}),\\x'\in h^{-1}(\{z'\})}} q^{}_{X}(x,x') \, \pi^{}_{X}(x), \quad z' \neq z, \; z, z' \in F,
	\end{split}
\end{equation}
and the corresponding rates of the mapped process are 
\begin{equation}\label{eq:marginal_q}
	q^{}_Z(z,z') = \frac{f_Z(z,z')}{\pi^{}_{X}(h^{-1}(\{z\}))}
\end{equation}
if $\pi^{}_{X}(h^{-1}(\{z\}))>0$. The quantities $f_Z(z,z')$ and $q^{}_Z(z,z')$ agree with the standard notions of fluxes and rates if $Z$ itself is Markov, but we would like to emphasise that we do not assume this here.

After these preparations, we can provide explicit expressions for the fluxes and rates of the pair $(\vec L^N,A^N)$ and of $A^N$ at stationarity. Here, $\mathbb{A}$ plays the role of the Markov process $X$, with the marginals $(\vec L^N, A^N)$ and $A^N$ as mapped processes.

\begin{theorem}\label{thm:fluxes_rates_finite}
	In the marginal process $({\vec L^N, A^N})$ at stationarity, the probability fluxes involving a beneficial ($1 \to 0$) or deleterious ($0 \to 1$) mutation on the ancestral line are given by
	\begin{equation}\label{finite_LA_benflux}
		f_{({\vec L^N, A^N})}\big ( (n,1),(n+k,0) \big ) = u^N \nu_0 w^N_{n+k} b^N_n , \quad n \in [N], \; k \in [N-n]_0,
	\end{equation}
	and
	\begin{equation}\label{finite_LA_delflux}
		f_{({\vec L^N, A^N})}\big ( (n,0),(n,1) \big ) = u^N \nu_1 w^N_{n} (b^N_{n-1}- b^N_n), \quad n \in [N].
	\end{equation}
	The corresponding transition rates thus read
	\begin{equation}\label{finite_LA_benrate}
		q_{(\vec L^N, A^N)}((n,1),(n+k,0)) = u^N \nu_0 \frac{w^N_{n+k}}{w^N_n}
	\end{equation}
	and
	\begin{equation}\label{finite_LA_delrate}
		\quad q_{(\vec L^N, A^N)}((n,0),(n,1)) = u^N \nu_1 \frac{b^N_{n-1}-b_n^N}{\sum_{j=1}^n (b^N_{j-1}-b_j^N)}.
	\end{equation}
	The mutation fluxes are 
	\begin{equation}\label{finite_fluxes}
		f_{A^N}(1,0) = u^N \nu_0 \sum_{n=1}^{N} a^N_{n-1} b^N_n =
		u^N \nu_1 \sum_{n=1}^N w^N_n (b^N_{n-1}-b^N_n) = f_{A^N}(0,1) ,
	\end{equation}
	and the mutation rates read
	\begin{equation}\label{finite_rates}
		q^{}_{A^N}(1,0) = u^N\nu_0\frac{\sum_{n=1}^N a^{N}_{n-1} b^N_n}{\sum_{k=1}^N w^{N}_k b^N_k } \quad \text{and}\quad
		q^{}_{A^N}(0,1) = u^N\nu_1\frac{\sum_{n=1}^N w^{N}_n (b^N_{n-1}-b^N_n )}{\sum_{k=1}^N a^{N}_{k-1} ( b^N_{k-1}-b^N_k )}. 
	\end{equation}
\end{theorem}

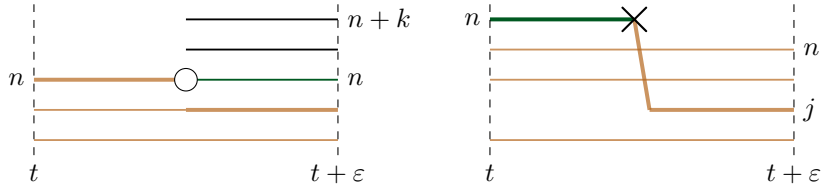
\begin{figure}
	\begin{tikzpicture}
	
	\draw[line width=0.7pt, lightbrown,opacity=0.7] (-5,0)--(-1,0);
	\draw[line width=0.7pt, lightbrown,opacity=0.7] (-5,0.4)--(-3,0.4);
	\draw[line width=1.5pt, lightbrown,opacity=0.7] (-3,0.4)--(-1,0.4);
	\draw[line width=1.5pt, lightbrown,opacity=0.7] (-5,0.8)--(-3,0.8);
	\draw[line width=0.7pt, darkgreen] (-3,0.8)--(-1,0.8);
	\draw[line width=0.7pt] (-3,1.2)--(-1,1.2);
	\draw[line width=0.7pt] (-3,1.6)--(-1,1.6);
	
	\draw[dashed] (-5,-0.2)--(-5,1.8);
	\draw[dashed] (-1,-0.2)--(-1,1.8);
	
	\node[below] at (-5,-0.2) {\(t\)};
	\node[below] at (-1,-0.2) {\(t+\varepsilon\)};
	\node[right] at (-1,0.8) {\(n\)};
	\node[left] at (-5,0.8) {\(n\)};
	\node[right] at (-1,1.6) {\(n+k\)};

	\node[draw,circle,inner sep=3pt, fill=white!100] at (-3,0.8) {};

	\draw[line width=1.5pt, darkgreen] (1,1.6)--(2.9,1.6);
	\draw[line width=1.5pt, lightbrown,opacity=0.7] (2.9,1.6)--(3.1,0.4);
	\draw[line width=1.5pt, lightbrown,opacity=0.7] (3.1,0.4)--(5,0.4);
	
	\draw[line width=0.7pt, lightbrown,opacity=0.7] (1,1.2)--(5,1.2);
	\draw[line width=0.7pt, lightbrown,opacity=0.7] (1,0.8)--(5,0.8);
	\draw[line width=0.7pt, lightbrown,opacity=0.7] (1,0)--(5,0);
	
	\draw[dashed] (1,-0.2)--(1,1.8);
	\draw[dashed] (5,-0.2)--(5,1.8);
	
	\node[below] at (1,-0.2) {\(t\)};
	\node[below] at (5,-0.2) {\(t+\varepsilon\)};
	\node[right] at (5,1.2) {\(n\)};
	\node[left] at (1,1.6) {\(n\)};
	\node[right] at (5,0.4) {\(j\)};

	\node at (2.9,1.6) {\scalebox{2}{\(\times\)}};
	
\end{tikzpicture}
	\caption{\label{fig_mut_ancline} Beneficial and deleterious mutations on the ancestral line. Dark green: type~0; light brown: type~1; bold: immune line.}
\end{figure}

The proof of the equality of the fluxes in \eqref{finite_fluxes} relies on the following simple general fact, which we prove here for convenience and lack of reference.

\begin{lemma}\label{lem:genfact}
	Let $Z=(Z_t)_{t \geq 0}$ be a stationary $\{0,1\}$-valued process whose sequence of jump points has no accumulation point almost surely. Then $f_Z(0,1)=f_Z(1,0)$, where $f_Z(0,1)$ and $f_Z(1,0)$ denote the probabilty fluxes of $Z$ from 0 to 1 and from 1 to 0, respectively.
\end{lemma}
\begin{proof}
	It is clear that, for any $\varepsilon>0$,
	\begin{equation*}
		\P(Z_t=0) = \P(Z_t=0,Z_{t+\varepsilon}=0) + \P(Z_t=0,Z_{t+\varepsilon}=1)
	\end{equation*}
	and
	\begin{equation*}
		\P(Z_{t+\varepsilon}=0) = \P(Z_t=0,Z_{t+\varepsilon}=0) + \P(Z_t=1,Z_{t+\varepsilon}=0).
	\end{equation*}
	The almost sure absence of an accumulation point for the sequence of jump points allows us to write
	\begin{equation*}
		\P(Z_t=0,Z_{t+\varepsilon}=1) = \varepsilon f_Z(0,1) + \scO(\varepsilon)
	\end{equation*}
	and 
	\begin{equation*}
		\P(Z_t=1,Z_{t+\varepsilon}=0) = \varepsilon f_Z(1,0) + \scO(\varepsilon).
	\end{equation*}
	But $\P(Z_t=0)=\P(Z_{t+\varepsilon}=0)$ due to stationarity, so $\varepsilon f_Z(0,1) = \varepsilon f_Z(1,0) + \scO(\varepsilon)$. Dividing by $\varepsilon$ and letting $\varepsilon$ to 0 yields the claim.
\end{proof}

Apart from this, our proof of Theorem \ref{thm:fluxes_rates_finite} will be based on the graphical representation and rely on the following crucial observation. It is obvious from Figures \ref{fig_pldASG_finite} and \ref{fig_mut_ancline} that the ancestral line can display a mutation only if, right before the mutation (in the forward direction, that is, reading the figures from left to right), the ancestral line is at the top and coincides with the immune line. Indeed, the pruning of all lines above a beneficial mutation in the backward direction, and the induced movement of the immune line to the top, means that, forward in time, beneficial mutations can only appear on a top line that is immune. And the only deleterious mutations on the ancestral line are those appearing on the immune line. But a deleterious mutation on the immune line induces movement of the latter to the highest level (in the backward direction), which means that deleterious mutations, too, can only be observed on an immune top line when going forward. With this observation in mind we are ready for the proof.

\begin{proof}[Proof of Theorem \ref{thm:fluxes_rates_finite}]
	Let $\varepsilon > 0$ and consider the event $\{\vec L^N_t=n, A^N_t=1, \vec L^N_{t+\epsilon}=n+k, A^N_{t+\varepsilon}=0\}$ for some $n \in [N]$ and $k \in [N-n]_0$, as shown in the left panel of Figure \ref{fig_mut_ancline}. Up to events with probability of order $\scO(\varepsilon)$, this event happens whenever $\vec L^N _{t+\varepsilon}= n+k$ (that is, at the right of the picture; probability $w^N_{n+k}$), a beneficial mutation occurs on line $n$ between $t$ and $t+\varepsilon$ (probability $\varepsilon u^N \nu_0$) and all $n$ lines are unfit at the left of the picture (probability $b_n^N$). Indeed, since the immune line is at the top to the left of the mutation, this ensures that line $n$ is ancestral and has type $1$ to the left and type $0$ to the right of the event. The corresponding probability thus reads
	\begin{equation}\label{LA_ben_mut}
		\P(\vec L^N_{t+\epsilon}=n+k, A^N_{t+\varepsilon}=0, \vec L^N_{t}=n, A^N_{t}=1) = 
		\varepsilon u^N \nu_0 w^N_{n+k} b^N_n + \scO(\varepsilon);
	\end{equation}
	dividing by $\varepsilon$ and letting $\varepsilon$ to 0 gives \eqref{finite_LA_benflux}. In a similar way, the event $\{\vec L^N_{t+\epsilon}=n,{A^N_{t+\varepsilon}=1},\allowbreak {\vec L^N_t=k, A^N_t=0} \}$, as shown in the right panel of Figure \ref{fig_mut_ancline}, occurs when, up to events with probability of order $\scO(\varepsilon)$, $\vec L^N_{t+\varepsilon}=n$ (probability $w^N_n$), a deleterious mutation occurs on the immune line between $t$ and $t+\varepsilon$ (probability $\varepsilon u^N \nu_1$), and line $n$ to the left of the event is the lowest type-0 level (probability $b_{n-1}^N-b_n^N$). Indeed, since the immune line is relocated to the top due to the mutation, this combination of events ensures that line $n$ at the left is ancestral and has type 0 to the left and type 1 to the right of the event. The corresponding probability thus reads
	\begin{equation}\label{LA_del_mut}
		\P(\vec L^N_{t+\epsilon}=n, A^N_{t+\varepsilon}=1, \vec L^N_{t}=n, A^N_{t}=0) = \varepsilon u^N \nu_1 w^N_n (b^N_{n-1}- b^N_n) + \scO(\varepsilon);
	\end{equation}
	dividing by $\varepsilon$ and letting $\varepsilon$ to 0 gives \eqref{finite_LA_delflux}. Eqs.~\eqref{finite_LA_benrate} and \eqref{finite_LA_delrate} follow immediately from \eqref{finite_LA_benflux} and \eqref{finite_LA_delflux} via \eqref{eq:marginal_q} since $\P(\vec L^N_t=n, A^N_t=1)=w^N_n b^N_n$ and $\P(\vec L^N_t=n, A^N_t=0)=w^N_n \sum_{j=1}^n (b^N_{j-1}-b^N_j)$, respectively. We get to the mutation fluxes by marginalising over $n$ (and $k$) in \eqref{LA_ben_mut} and \eqref{LA_del_mut} using \eqref{eq:marginal_f}; the fluxes are equal due to Lemma~\ref{lem:genfact}. They lead to the rates via Eqns. \eqref{anc_unfit}, \eqref{anc_fit}, and \eqref{eq:marginal_q}.
\end{proof}

It is important to emphasise that the rates $q^{}_{A^N}(1,0)$ and $q^{}_{A^N}(0,1)$ are averaged over $\vec L^N$. And let us mention an alternative argument that leads to the mutation fluxes via the forward view with an aspect of projection into the future. Namely, conditional on $Y^N_t=k$, any of the $k$ unfit individuals turns into a fit one at rate $u^N\nu_0$, and then it is ancestor with the probability \eqref{good_ind_anc}; the latter automatically implies that the unfit predecessor is also ancestral. We therefore get
\begin{align*}
	f_{A^N}(1,0) & = u^N \nu_0 \sum_{k=1}^N k \sum_{n=1}^N a^N_{n-1} \frac{(k-1)^{\underline {n-1}}}{N^{\underline {n}}} \P(Y^N_t=k) \\
	& = u^N \nu_0 \sum_{n=1}^N a^N_{n-1} b^N_n.
\end{align*}
Likewise, with \eqref{bad_ind_anc} for deleterious mutations,
\begin{align*}
	f_{A^N}(0,1) & = u^N \nu_1 \sum_{k=1}^N (N-k) \sum_{n=1}^N w^N_{n} \frac{k^{\underline {n-1}}}{N^{\underline {n}}} \P(Y^N_t=k) \\
	& = u^N \nu_1 \sum_{k=1}^N \sum_{n=1}^N w^N_{n} \frac{k^{\underline {n-1}}}{N^{\underline {n-1}}} \frac{N-k}{N-n+1}\P(Y^N_t=k) \\
	& = u^N \nu_1 \sum_{n=1}^N w^N_{n} \big ( b^N_{n-1} - b^N_n \big ).
\end{align*}

\subsection{Diffusion limit}
Since the process $\cL$ is once more positive recurrent, the notion of common ancestor is again well posed in the diffusion limit \citep{len}, and in the long term, the common ancestor agrees with the true ancestor in our construction. Let us now consider the fluxes and rates along the (common) ancestor line in this limit, which are defined as ${f_{\cA}(1,0)\defeq \lim_{N \to \infty} N f_{A^N}(1,0)}$, $q^{}_{\cA}(1,0)\defeq \lim_{N \to \infty} N q_{A^N}(1,0)$ (due to the rescaling of time), and likewise for $f_{\cA}(0,1)$ and $q_{\cA}(0,1)$. As a simple consequence of Theorem \ref{thm:fluxes_rates_finite}, we obtain
\begin{coro} \label{coro:diff_marg_f_q}	
The mutation fluxes and rates of the common ancestor type process in the diffusion limit at stationarity read
	\begin{equation*}
		\begin{split}
			f_{\cA}(1,0) &= \theta \nu_0 \sum_{n>0} \alpha_{n-1} \beta_n, \\
			f_{\cA}(0,1) &= \theta \nu_1 \sum_{n>0} \omega_{n} (\beta_{n-1}-\beta_n), \\
			q^{}_{\cA}(1,0) &= \theta \nu_0\frac{\sum_{n>0} \alpha_{n-1} \beta_n}{\sum_{m>0}\omega_m \beta_m}, \quad \text{and}\\
			q^{}_{\cA}(0,1) &= \theta \nu_1\frac{\sum_{n>0}(\alpha_{n-1}-\alpha_n)(\beta_{n-1}-\beta_n )}{\sum_{m>0} \alpha_{m-1}(\beta_{m-1}-\beta_m)}
		\end{split} 
	\end{equation*}
	with $\alpha_n$, $\beta_n$, and $\omega_n$ as in Section~\ref{sec:difflimit}. 
\end{coro}
\begin{proof}
	The explicit expressions follow from \eqref{finite_fluxes} and \eqref{finite_rates} together with the convergence of $Nu^N$ to $\theta$ and the $a^N_n$ and $b^N_n$ to $\alpha_n$ and $\beta_n$, respectively.
\end{proof}

Note that the mutation rates in the corollary were already discovered by \cite[Corollary 2]{Fearnhead} by mainly analytical methods; rederiving them here on the grounds of the pLD-ASG greatly simplifies the calculation and provides a lucid and intuitive meaning.

\subsection{Deterministic limit}
In contrast to the diffusion limit and the finite case, a common ancestor of the entire population does not exist in the deterministic limit; this is due to the absence of coalescence events. Nevertheless, the ancestral line of a randomly-chosen individual continues to be well defined. Let us now consider the rates along this line, which are defined as $q^{}_{A}(1,0)\defeq \lim_{N \to \infty} q_{A^N}(1,0)$, and likewise for $q_{A}(0,1)$. They once more follow from Theorem \ref{thm:fluxes_rates_finite}.

\begin{coro}\label{mutrates_det} 
	The mutation rates of the type process on the ancestral line in the deterministic limit at stationarity read
	\begin{equation*}
		q^{}_{A}(1,0) = u\nu_0\frac{1}{1-p} \quad \text{and } \; 
		q^{}_{A}(0,1) = u\nu_1(1-p)
	\end{equation*}
	with $p$ of \eqref{eq_p}.
\end{coro}
\begin{proof}
	The result is immediate if we take the limits in \eqref{finite_rates} and use \eqref{DistrConvDet} as well as \eqref{geo}:
	\begin{align*}
		q^{}_{A}(1,0) & = u\nu_0\frac{\sum_{n=1}^{\infty}	a_{n-1} y_\infty^{n}}{\sum_{m=1}^\infty w_m y_\infty^m} = u\nu_0\frac{1}{1-p} \quad \text{and} \\
		q^{}_{A}(0,1) & = u\nu_1\frac{\sum_{n=1}^{\infty} w_n	y_\infty^{n-1}(1-y_\infty)}{\sum_{m=0}^\infty a_{m} y_\infty^{m}(1-y_\infty)} = u\nu_1 (1-p)
	\end{align*}
	with $a_n$, $b_n$, and $w_n$ as in Section~\ref{sec:detlimit}. 
\end{proof}
\begin{remark}\label{georgii}
	Let us mention that the resulting process is similar in spirit to the retrospective process of \citet[Definition~3.1]{Georgii}, which describes the mutation process on the ancestral line of a typical individual in a multitype branching process. Despite the different model, the retrospective process (when specialised to 2 types)  is closely related to the one described here, since $g(y_\infty)/y_\infty$ and $(1-g(y_\infty))/(1-y_\infty)$ of \eqref{expoff} have the meaning of the long-term offspring amount of a type-1 and a type-0 individual, respectively, in a stationary population.
\end{remark}

\section{Flux identities}
\label{section_fluxes}
It is clear (and has already been mentioned in Theorem~\ref{thm:fluxes_rates_finite}) that, at stationarity, the mean beneficial and deleterious mutation fluxes on the ancestral line balance each other. However, this is anything but obvious from the explicit expressions. Specifically, it turns out that the equality in \eqref{finite_fluxes} does, in general, not hold term by term. That is, for the fluxes $f^{(n)}_{A^N}(1,0)$ and $f^{(n)}_{A^N}(0,1)$ of beneficial and deleterious mutations at level $n$, we have
\begin{equation}\label{partial_fluxes}
	f^{(n)}_{A^N}(1,0) \defeq u^N \nu^{}_0 a^{N}_{n-1} b^N_n \neq u^N \nu^{}_1 w^{N}_n (b^N_n-b^N_{n-1} ) \eqdef f^{(n)}_{A^N}(0,1), \quad n>0, 
\end{equation}
as illustrated in Figure~\ref{fig:partial_fluxes}. Hence, the beneficial and deleterious mutation fluxes do not balance each other on every individual level. It is therefore interesting to investigate these individual fluxes. 

\begin{prop}[flux identities]\label{ind_fluxes}
	In the finite system, the individual fluxes at stationarity satisfy
	\begin{equation*}
		f^{(n)}_{A^N}(1,0) + \sum_{i=n+1}^N w^{N}_i (b^N_{i-1}-b^N_i) = f^{(n)}_{A^N}(0,1) + (n-1) w^{N}_n (b^N_{n-1}-b^N_{n} ), \quad n>0.
	\end{equation*}
\end{prop}
\begin{figure}
	\includegraphics[width=0.6\textwidth]{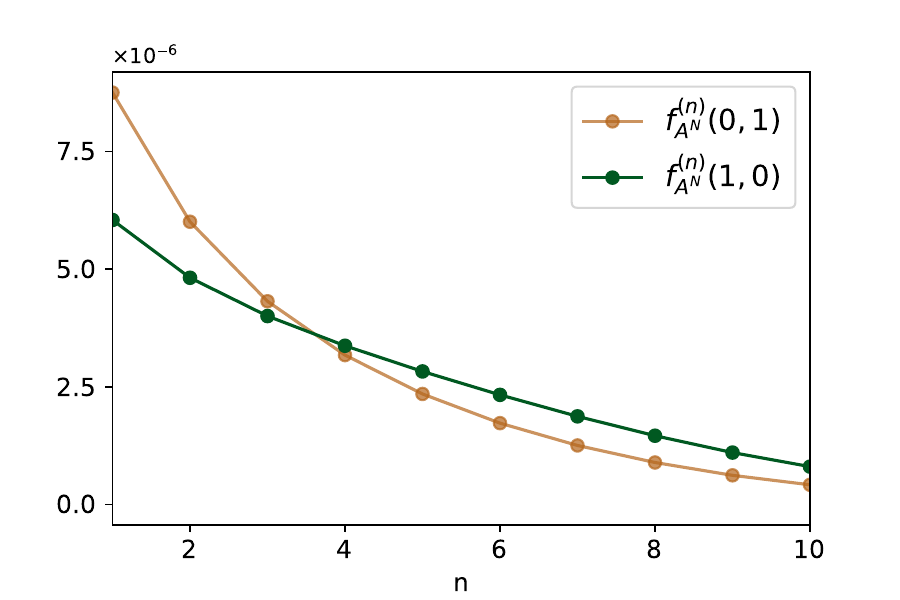}\caption{\label{fig:partial_fluxes} The individual mutation fluxes of \eqref{partial_fluxes}. Parameters: ${N=10^4},\allowbreak s^N=1.5 \cdot 10^{-3}, u^N=8 \cdot 10^{-4}, \nu_1=0.99$.}
\end{figure}
\begin{proof}
	We start from the recursions \eqref{rec_a} and \eqref{rec_b} and rewrite them as
	\begin{align}
		\Big (\frac{i+1}{N}+u^N \nu_1\Big)(a_i^N-a_{i+1}^N) &= s^N\frac{N-i}{N}(a_{i-1}^N-a_i^N) - u^N \nu_0 a_{i}^N,\ \quad 0 < i < N, \label{rec_a_diff}\\
		\Big (\frac{i-1}{N}+u^N \nu_1\Big )(b_{i-1}^N-b_{i}^N) &= s^N\frac{N-i}{N}(b_{i}^N-b_{i+1}^N) + u^N \nu_0b_{i}^N, \quad 0 < i \leq N. \label{rec_b_diff}
	\end{align}
	We now introduce 
	\begin{equation*}
		S_n \defeq \sum_{i=n}^N(a_{i-1}^N-a_{i}^N)(b_{i-1}^N-b_{i}^N).
	\end{equation*}
	We then have
	\begin{equation*}
		\sum_{n=1}^N S_{n+1} = \sum_{i=1}^N(i-1)(a_{i-1}^N-a_{i}^N)(b_{i-1}^N-b_{i}^N).
	\end{equation*}
	Using \eqref{rec_a_diff} in the first step and \eqref{rec_b_diff} in the last one, we get for $i \geq 1$:
	\begin{align*}
		&\Big[u^N \nu_0 a_i^N + \Big(\frac{i+1}{N}+u^N\nu_1\Big)(a_i^N-a_{i+1}^N)\Big](b_i^N-b_{i+1}^N)+u^N \nu_0(a_{i-1}^N-a_i^N)b_i^N\\
		&\hspace{1cm}=s^N \frac{N-i}{N}(a_{i-1}^N-a_i^N)(b_i^N-b_{i+1}^N)+u^N\nu_0(a_{i-1}^N-a_i^N)b_i^N\\
		&\hspace{1cm}=(a_{i-1}^N-a_i^N)\Big[s^N \frac{N-i}{N}(b_i^N-b_{i+1}^N)+u^N\nu_0b_i^N\Big]\\
		&\hspace{1cm}=(a_{i-1}^N-a_i^N)(b_{i-1}^N-b_{i}^N)\Big(\frac{i-1}{N}+u^N \nu_1\Big).
	\end{align*}
	This can be rewritten as
	\begin{align*}
		&u^N\nu_0(a_{i-1}b_i^N-a_i^Nb_{i+1}^N)+\frac{1}{N}(a_i^N-a_{i+1}^N)(b_i^N-b_{i+1}^N)\\
		&\hspace{1cm}=(a_{i-1}^N-a_i^N)(b_{i-1}^N-b_i^N)\Big(\frac{i-1}{N}+u^N \nu_1\Big)-\Big(\frac{i}{N}+u^N \nu_1\Big)(a_i^N-a_{i+1}^N)(b_i^N-b_{i+1}^N).
	\end{align*}
	Summing over $i \in [k:N]$ on both sides leads to
	\begin{align*}
		&u^N \nu_0a_{n-1}^Nb_n^N-u^N \nu_0a_{N}^Nb^N_{N+1}+\frac{1}{N}S_{n+1}\\
		&\hspace{1cm}=(a_{n-1}^N-a_n^N)(b_{n-1}^N-b_n^N)\Big(\frac{n-1}{N}+u^N \nu_1\Big)-(a_{N}^N-a_{N+1}^N)(b_{N}^N-b_{N+1}^N)\left(1+u^N \nu_1\right);
	\end{align*}
	considering that $a_N^N=a_{N+1}^N=0=b^N_{N+1}$,
	this yields the claim.
\end{proof}
Note that the proposition says that the individual fluxes balance each other together with certain coalescence events, see Figure~\ref{fig:flux_id}. In detail, the flux \(f^{(n)}_{A^N}(1,0)\) of beneficial mutations at level \(n\) together with one-half the flux of coalescences related to arrows sent out by an unfit line \(n\) and pushing the first fit line upwards (forward in time) equals the flux \(f^{(n)}_{A^N}(0,1)\) of deleterious mutations at level \(n\) together with one-half the flux of coalescences that push line \(n\) upwards when it is the first fit line. Summing both sides in the statement of Proposition~\ref{ind_fluxes} over \(n\), we find:
\begin{align*}
	u^N \nu_0& \sum_{n=1}^Na_{n-1}^Nb_n^N \\ &=\frac{1}{N}\sum_{n=1}^N(n-1) w_n^N(b_{n-1}^N-b_n^N) - \frac{1}{N}\sum_{n=1}^N S_{n+1} +\sum_{n=1}^N u^N \nu_1w_n^N(b_{n-1}^N-b_n^N),
\end{align*}
which is the identity of the fluxes in Theorem~\ref{thm:fluxes_rates_finite}, as it must be.

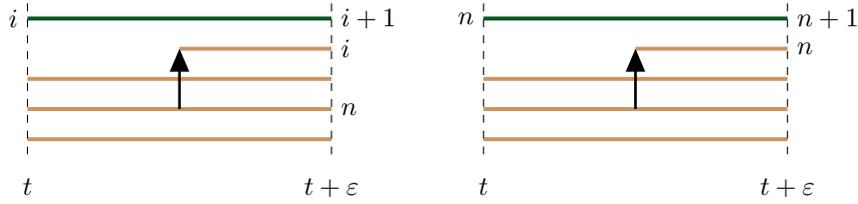
\begin{figure}
	\begin{tikzpicture}[scale=1]
	
	\draw [color = lightbrown,opacity=0.7,line width=1.5](-5,0)--(-1,0);
	\draw [color = lightbrown,opacity=0.7,line width=1.5](-5,0.4)--(-1,0.4);
	\draw [color = lightbrown,opacity=0.7,line width=1.5](-5,0.8)--(-1,0.8);
	\draw [color = lightbrown,opacity=0.7,line width=1.5](-3,1.2)--(-1,1.2);
	\draw [color = darkgreen,line width=1.5] (-5,1.6)--(-1,1.6);

	\draw [color = lightbrown,opacity=0.7,line width=1.5](1,0)--(5,0);
	\draw [color = lightbrown,opacity=0.7,line width=1.5](1,0.4)--(5,0.4);
	\draw [color = lightbrown,opacity=0.7,line width=1.5](1,0.8)--(5,0.8);
	\draw [color = lightbrown,opacity=0.7,line width=1.5](3,1.2)--(5,1.2);
	\draw [color = darkgreen,line width=1.5] (1,1.6)--(5,1.6);
	
	\draw[-{triangle 45},line width=1] (3,0.4) -- (3,1.2);
	\draw[-{triangle 45},line width=1] (-3,0.4) -- (-3,1.2);

	\node[below=0.4] at (-5,0) {\scalebox{1}{\(t\)}};
	\node[below=0.4] at (-1,0) {\scalebox{1}{\(t+\varepsilon\)}};
	\node[below=0.4] at (1,0) {\scalebox{1}{\(t\)}};
	\node[below=0.4] at (5,0) {\scalebox{1}{\(t+\varepsilon\)}};
	
	\node[left] at (-5,1.6) {\scalebox{1}{\(i\)}};
	\node[left] at (1,1.6) {\scalebox{1}{\(n\)}};
	\node[right] at (-1,1.2) {\scalebox{1}{\(i\)}};
	\node[right] at (5,1.2) {\scalebox{1}{\(n\)}};
	\node[right] at (-1,0.4) {\scalebox{1}{\(n\)}};
	
	\node[right] at (-1,1.6) {\scalebox{1}{\(i+1\)}};
	\node[right] at (5,1.6) {\scalebox{1}{\(n+1\)}};
	
	\draw[dashed] (-1,1.8)--(-1,-0.2);
	\draw[dashed] (-5,1.8)--(-5,-0.2);
	\draw[dashed] (1,1.8)--(1,-0.2);
	\draw[dashed] (5,1.8)--(5,-0.2);
	
\end{tikzpicture}
	\caption{\label{fig:flux_id} The coalescence events appearing in the flux identities of Proposition~\ref{ind_fluxes}. The corresponding mutation events are those of Figure~\ref{fig_mut_ancline}.}
\end{figure}

The following corollary is an immediate consequence of Proposition \ref{ind_fluxes}.
\begin{coro}
	In the diffusion limit at stationarity, the individual fluxes of Proposition \ref{ind_fluxes} turn into
	\begin{equation*}
		\theta \nu_0 \alpha_{n-1} \beta_n + \sum_{i>n} \omega_i(\beta_{i-1}-\beta_i) = \big (\theta \nu^{}_1 + (n-1) \big ) \omega_n (\beta_{n-1}-\beta_{n} ), \quad n > 0,
	\end{equation*}
	while in the deterministic limit at stationarity, they read
	\begin{equation*}
		u\nu_0 a_{n-1} b_n = u \nu_1 w_n (b_{n-1}-b_{n}), \quad n>0,
	\end{equation*}
	that is, the mutation fluxes do balance each other at every level in this case.
\end{coro}

\section{Parameter dependence}\label{sec_parameters}
The focal parameter is the selective parameter $s^N$, which is the branching rate in the backward time direction. We will determine here its effect on the tail probabilities, the sampling probabilites, and the mutation rates and fluxes on the ancestral line.

\subsection{The tail probabilities}
\begin{prop}\label{prop:tailprobs}
	The tail probabilities $a^N_n$ (for $n \in [N]$), $\alpha^{}_n$ (for $n \in \N$), and $a^{}_n$ (for $n \in \N$) are strictly increasing in $s^N, \sigma$, and $s$, respectively.
\end{prop}
\begin{proof}
	We start with the finite case. To this end, we consider two copies of $L^N$, namely $L^{N,1}$ and $L^{N,2}$ with branching rates $s^{N,1}$ and $s^{N,2} > s^{N,1}$, respectively, and couple them on the same probability space. More precisely, we consider the process $(\overline L^{N,1},\overline L^{N,2}) $ on $[N] \times [N]$ with transitions
	\begin{align*}
		&(n,m) \to (n+1,m+1) & \text{at rate} &&& s^{N,1} \min \Big \{n \frac{N-n}{N}, m \frac{N-m}{N} \Big \}, \\
		&(n,m) \to (n,m+1) & \text{at rate} &&& s^{N,1} \Big ( m \frac{N-m}{N} - n \frac{N-n}{N} \Big ) \one_{\{ m \frac{N-m}{N} > n \frac{N-n}{N}\}} \\
		&&&&&\quad + (s^{N,2}-s^{N,1}) m \frac{N-m}{N}, \\
		&(n,m) \to (n+1,m) & \text{at rate} &&& s^{N,1} \Big ( n \frac{N-n}{N} - m \frac{N-m}{N} \Big ) \one_{\{n \frac{N-n}{N} > m \frac{N-m}{N}\}}, \\ 
		&(n,m) \to (n-1,m-1) & \text{at rate} &&& \min \Big \{ n \frac{n-1}{N}+ u^N \nu_1 (n-1), m \frac{m-1}{N}+ u^N \nu_1 (m-1) \Big \}, \\ 
		&(n,m) \to (n-1,m) & \text{at rate} &&& \Big (n \frac{n-1}{N}+ u^N \nu_1 (n-1) \Big ) \one_{\{n>m\}}, \\ 
		&(n,m) \to (n,m-1) & \text{at rate} &&& \Big (m \frac{m-1}{N}+ u^N \nu_1 (m-1) \Big ) \one_{\{m>n\}}, \\ 
		&(n,m) \to (j,j) & \text{at rate} &&& u^N \nu_0 \one_{\{0<j<\min\{n,m\}\}}, \\ 
		&(n,m) \to (n,j) & \text{at rate} &&& u^N \nu_0 \one_{\{n\leq j < m\}}, \\ 
		&(n,m) \to (j,n) & \text{at rate} &&& u^N \nu_0 \one_{\{m\leq j < n\}}
	\end{align*}
	for all $n,m \in [N]$; no other transitions are possible (with probability 1). Note that target states outside $[N]\times[N]$ are included in the list, but are attained at rate 0, so may be ignored. Note also that the transitions are not mutually exclusive: for example, $(n,n) \to (n-1,n-1)$ may either happen due to coalescence or deleterious mutation, or due to beneficial mutation; so the total rate is $n \frac{n-1}{N}+ u^N \nu^{}_1 (n-1) + u^N \nu^{}_0$. We have chosen this presentation of the rates to avoid somewhat technical case distinctions. 
	
	It is easily verified that the marginal rates corresponding to $\overline L^{N,1}$ and $\overline L^{N,2}$, respectively, equal those of $L^{N,1}$ and $L^{N,2}$, so $(\overline L^{N,1},\overline L^{N,2})$ is indeed a coupling of $L^{N,1}$ and $L^{N,2}$. There are no transitions $(n,m) \to (k,j)$ for $m \geq n$ and $j < k$, so $S \defeq \{ (n,m) : n,m \in [N], m \geq n\}$ is an invariant set for $(\overline L^{N,1},\overline L^{N,2}) $; hence the coupling is monotone. Furthermore, the transitions $(n,m) \to (1,1)$, $(n,m) \to (n+1,m+1)$, and $(n,m) \to (n,m+1)$ occur at positive rates for all $n,m \in [N]$, so $(\overline L^{N,1},\overline L^{N,2}) $ is irreducible and positive recurrent on $S$. It therefore has a stationary distribution that assigns positive weight to every state in $S$. Let $\overline L^{N,1}_\infty,\overline L^{N,2}_\infty$ be random variables on $S$ that have this stationary distribution. They obviously satisfy $\P(\overline L^{N,2}_\infty \geq \overline L^{N,1}_\infty)=1$ and $\P(\overline L^{N,2}_\infty> \overline L^{N,1}_\infty)>0$. We may thus conclude that, for $n \in [N]_0$,
	\begin{equation*}
		\begin{split}
			 a^{N,2}_n \defeq \P(L_\infty^{N,2}>n) &= \P(\overline L^{N,2}_\infty>n) > \P(\overline L^{N,2}_\infty \geq \overline L^{N,1}_\infty>n) \\
			 &= \P(\overline L^{N,1}_\infty>n) = \P(L_\infty^{N,1}>n) \defeq a^{N,1}_n, 
		\end{split}
	\end{equation*}
	which proves the claim in the finite setting. In both the diffusion and the deterministic limits, the proofs are analogous with obvious modifications. (In the deterministic limit, one may alternatively invoke the fact that $a^{}_n =p^n$ and $p$ is strictly increasing in $s$, see Proposition~\ref{prop:derivatives_det} below.)
\end{proof}

\subsection{The sampling probabilities}
We continue with the sampling probabilities, whose behaviour is unsurprising, as we have already seen in Figure~\ref{fig:anc_dist}. We prove it here for convenience and lack of reference.
\begin{lemma}\label{lem:samplingprobs}
	The sampling probabilities at stationarity, that is $b^N_n$ (for $n \in [N]$), $\beta^{}_n$ (for $n \in \N$), and $b^{}_n$ (for $n \in \N$), are strictly decreasing in $s^N, \sigma$, and $s$, respectively.
\end{lemma}
\begin{proof}
	We start with the finite case, work with the line-counting process $R^N$ of the killed ASG as introduced in Appendix~A (Section~\ref{sec:app_A}), and use the representation of $b_n^N$ as the absorption probability of $R^N$ as presented in Corollary~\ref{coro:b_n}. In close analogy with the proof of Proposition \ref{prop:tailprobs}, we consider two copies of $R^N$, namely $R^{N,1}$ and $R^{N,2}$ with $R^{N,1}_0=R^{N,2}_0=n \in [N]$ and branching rates $s^{N,1}$ and $s^{N,2} > s^{N,1}$, respectively, and couple them on the same probability space. Because of the similarity with the previous proof, we do not spell out all the transitions; suffice it to say that one obtains a monotone coupling $(\overline R^{N,1},\overline R^{N,2}) $ on $[N]_{0,\Delta} \times [N]_{0,\Delta}$, where we identify $\Delta$ with $N+1$. The marginals equal $R^{N,1}$ and $R^{N,2}$ in distribution. Since $q_{R^{N,2}}(n,n+1) > q_{R^{N,1}}(n,n+1)$ for every $n \in [N-1]$, it is clear that, while $\P(\overline R_t^{N,1} \leq \overline R_t^{N,2} )=1$ for all $t \geq 0$ by the monotonicity of the coupling, one also has $\P(\overline R_t^{N,1} <\overline R_t^{N,2} ) > 0$ for all $t>0$. Let $\tau$ be the (random) time of absorption of $\overline R^{N,2}$. If $\overline R^{N,2}_\tau=0$, we also have $\overline R^{N,1}_\tau=0$. If $\overline R^{N,2}_\tau=\Delta$ (which happens with positive probability), we have $\P(\overline R^{N,1}_\tau<\Delta)>0$; this entails an additional chance for $\overline R^{N,1}$ to absorb in 0. More precisely, we altogether have, with $R^{N,i}_\infty \defeq \lim_{r \to \infty} R^{N,i}_r$ for $i \in \{1,2\}$ and likewise for $\overline R^{N,i}_\infty$:
	\begin{equation*}
		\begin{split}
			 \P(R^{N,1}_\infty = 0) &= \P(\overline R^{N,1}_\infty = 0) \\
			 &= \P(\overline R^{N,2}_\infty = 0) 
			 + \P(\overline R^{N,2}_\infty = \Delta) \P(\overline R^{N,1}_\infty =0 \mid \overline R^{N,2}_\infty = \Delta) 
			 > \P(R^{N,2}_\infty =0),
		\end{split}
	\end{equation*}
	which shows the claim in the finite case. In both the diffusion and the deterministic limits, the proofs are analogous with obvious modifications. 
\end{proof}

\subsection{The ancestral type distribution}
We now turn to the ancestral type distribution, whose dependence on the selection parameter is illustrated in Figure~\ref{fig:anc_dist}. In what follows, we use the shorthand $x'$ for $\frac{\dd}{\dd s^N} x(s^N)$, $\frac{\dd}{\dd \s} x(\s)$ or $\frac{\dd}{\dd s} x(s)$ for a quantity $x$ depending on \(s^N,\s\) and \(s\), respectively (there will be no risk of confusion) and suppress the explicit parameter dependence.
\begin{prop}\label{prop:ancdist}
	The stationary probabilities for a type-1 ancestor, namely $\P(A^N_\infty=1)$, \allowbreak$\P(\cA_\infty=1)$, and $\P(A_\infty=1)$ from \eqref{anc_unfit}, \eqref{anc_unfit_diff}, and \eqref{anc_unfit_det}, are strictly decreasing in $s^N, \sigma$, and $s$, respectively.
\end{prop}
\begin{proof}
	For the finite case, differentiating $\P(A^N_\infty=1)$ with respect to $s^N$ gives
	\begin{equation}\label{anc_unfit_prime}
		\frac{\dd}{\dd s^N} \P(A^N_\infty=1) = \frac{\dd}{\dd s^N} \sum_{n=1}^N w^{N}_n b^N_n = \sum_{n=1}^N (w^{N}_n)' b^N_n + \sum_{n=1}^N w^{N}_n (b^N_n)'.
	\end{equation}
	From Lemma~\ref{lem:samplingprobs}, we know that $(b^N_n)'<0$ for all $n \in [N]$, so the second sum is negative. Write the first sum as
	\begin{equation*}
		\sum_{n=1}^N (w^{N}_n)' b^N_n = \E[h(Y^N_{\infty})]
	\end{equation*}
	with
	\begin{equation*}
		h(k) \defeq \sum_{n=1}^N (w^{N}_n)' \frac{k^{\underline{n}}}{N^{\underline{n}} }
		= \frac{\dd}{\dd s^N} \Big ( \sum_{n=1}^N w^{N}_n \frac{k^{\underline{n}}}{N^{\underline{n}}} \Big )
		= \frac{\dd}{\dd s^N} \E \Big ( \frac{k^{\underline{L^N_\infty}}}{N^{\underline{L^N_\infty}}} \Big ) < 0,
	\end{equation*}
	where the inequality is due to the stochastic monotonicity of $L^N_\infty$ established in Proposition~\ref{prop:tailprobs}. It follows that the first sum in \eqref{anc_unfit_prime} is negative as well, which establishes the claim in the finite case. The proofs for the diffusion and deterministic limits are perfectly analogous. In the diffusion limit, the result has been shown previously by \citet[Prop. 2.6]{Taylor07} via a different route.
\end{proof}

\subsection{The mutation rates}
A first inspection of \eqref{finite_rates} shows the following. For $s^N=0$, we have $a^N_n=\delta_{0,n}$ and $w^N_n=\delta_{1,n}$, where $\delta$ denotes Kronecker's delta; this entails that
\begin{equation*}
	\frac{\sum_{n=1}^N a^{N}_{n-1} b^N_n}{\sum_{k=1}^N w^{N}_k b^N_k }=1=\frac{\sum_{n=1}^N w^{N}_n (b^N_{n-1}-b^N_n )}{\sum_{k=1}^N a^{N}_{k-1} ( b^N_{k-1}-b^N_k )}
\end{equation*}
and so $q^{}_{A^N}(1,0)=u^N\nu^{N}_0$ and $q^{}_{A^N}(0,1)=u^N\nu^{N}_1$ --- as is to be expected, since in this case the mutation rates on the ancestral line should equal the neutral rates. In contrast, for $s^N > 0$, we have $a^N_n>0$ and $w^N_n>0$ with $a^N_{n-1} = w^N_n + a^N_n > w^N_n$ for all $n>0$, so $q^{}_{A^N}(1,0)> u^N\nu^{N}_0$ and $q^{}_{A^N}(0,1)< u^N\nu^{N}_1$. See Fig.~\ref{fig:mutrates} for a numerical evalution of the dependence of $q^{}_{A^N}(1,0)$ and $q^{}_{A^N}(0,1)$ on $s^N$. Indeed, the figure tells us more than our quick inspection: $q^{}_{A^N}(1,0)$ and $q^{}_{A^N}(0,1)$ are strictly increasing and decreasing, respectively. While we only have the numerical evidence for finite $N$, we now give proofs for our two limits. They will only require results from Sections~\ref{sec_model}--\ref{section_backward}; in particular, they do not rely on the flux identities.
\begin{figure}
	\includegraphics[width=0.6 \textwidth]{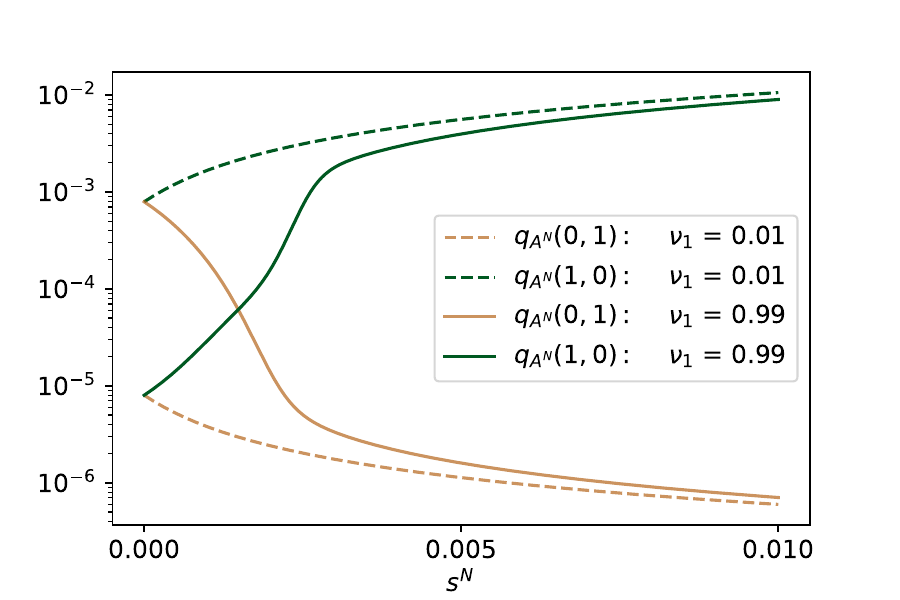}
	\caption{\label{fig:mutrates} Mutation rates, on log scale, in the finite system as functions of $s^N$ for two values of $\nu_1$. Further parameters: $N=10^4, u^N=8 \cdot 10^{-4}$.}
\end{figure}
\begin{theorem}\label{thm:derivatives}
	The mutation rates on the ancestral line in the diffusion limit at stationarity satisfy
	\begin{align}
		\label{mu10'}
		q_{\cA}'(1,0) & = \frac{\theta\nu^{}_0 b_1}{\big ( \sum_{m>0} (\alpha^{}_{m-1}-\alpha^{}_m) \beta_m \big )^2}\sum_{n > 0} \alpha'_n \b_n > 0 ,\\[2mm]
		\label{mu01'}
		q_{\cA}'(0,1) & = - \frac{\theta\nu^{}_1 (\beta_0-\beta_1)}{\big ( \sum_{m>0} \alpha^{}_{m-1} (\beta_{m-1}-\beta_m) \big )^2} \sum_{n > 0} \a'_n (\beta_{n-1}-\beta_n) < 0.
	\end{align} 
\end{theorem}

The proof requires two lemmas.

\begin{lemma}\label{lem:alphabeta}
	The tail probabilities of $L_\infty$ and the sampling probabilities are connected via
	\begin{equation*}
		\alpha^{}_n = \frac{\sigma^n}{\prod_{j=1}^n (j+\theta \nu^{}_1)} \frac{\beta_{n+1}-\beta_{n+2}}{\beta_1-\beta_2}, \quad n>0.
	\end{equation*}
\end{lemma}
Note that the lemma connects two sets of coefficients that characterise the present and the past. It is thus interesting on its own. The proof we give below is a simple verification. However, there is a deep probabilistic meaning behind the lemma, which is related to pairs of birth-death processes with catastrophes, the connection between the stationary distribution of one of the processes with the absorption probabilities of the other, and a generalised detailed-balance equation connecting the two. This will be treated in a forthcoming paper. 
\begin{proof}
	Set $\bar{\alpha}_0=1$ and
	\begin{equation*}
		\bar{\alpha}_n\defeq\frac{\sigma^n(\beta^{}_{n+1}-\beta^{}_{n+2})}{F_n(\beta^{}_1-\beta^{}_2)},\quad n>0,
	\end{equation*}
	where $F_n\defeq \prod^n_{j=1} (j+\vartheta\nu^{}_1)$. Since $\alpha\defeq(\alpha^{}_n)_{n\geq 0}$ is the unique solution to the system of equations \eqref{recursion} with boundary conditions $\alpha_0=1$ and $\lim_{n\to\infty}^{}\alpha_n=0$, it suffices to show that $\bar \a \defeq (\bar{\alpha}_n)_{n\geq 0}$ solves the same system of equations and boundary conditions. By definition $\bar{\alpha}^{}_0=1$. Moreover, since $\lim_{n\to\infty}\beta_n=0$ and $\lim_{n\to\infty}\sigma^n/F_n=0$, we deduce that $\lim_{n\to\infty}\bar{\alpha}_n=0$. It remains to show that $\bar{\alpha}$ satisfies \eqref{recursion}.
	
	Let $n>0$. Using the definition of $\bar{\alpha}_n$ and \eqref{rec_beta} first for $n+2$ and then for $n+1$, we obtain
	\begin{align*}
		(n+1+\s+\t)\bar{\alpha}_n &= (n+1+\s+\t)\frac{\s^n(\beta_{n+1}-\beta_{n+2})}{F_n(\beta_1-\beta_2)}\\
		&= \frac{\s^n}{F_n(\beta_1-\beta_2)}\left[(n+1+\s+\t)\beta_{n+1}-(\s \beta_{n+3}+(n+1+\tt)\beta_{n+1})\right]\\
		&= \frac{\s^n}{F_n(\beta_1-\beta_2)}\left[(n+\s+\t)\beta_{n+1}-(\s \beta_{n+3}+(n+\tt)\beta_{n+1})\right]\\
		&= \frac{\s^n}{F_n(\beta_1-\beta_2)}\left[\s \beta_{n+2}+(n+\tt)\beta_n-(\s \beta_{n+3}+(n+\tt)\beta_{n+1})\right]\\
		&= \frac{\s^n}{F_n(\beta_1-\beta_2)}\left[(n+\tt)(\beta_n-\beta_{n+1})+\s (\beta_{n+2}-\beta_{n+3})\right]\\
		&=\sigma\frac{\s^{n-1}(\beta_{n}-\beta_{n+1})}{F_{n-1}(\beta_1-\beta_2)}+(n+1+\tt)\frac{\s^{n+1}(\beta_{n+2}-\beta_{n+3})}{F_{n+1}(\beta_1-\beta_2)}\\
		&=\sigma \bar{\alpha}_{n-1}+(n+1+\tt)\bar{\alpha}_{n+1}.
	\end{align*}
	This ends the proof.
\end{proof}

\begin{lemma}\label{lem:alpha'beta'}
	The derivatives of the tail and sampling probabilities with respect to $\sigma$ read
	\begin{align}
		\alpha'_n & = \alpha^{}_{n-1} - K \alpha^{}_n \quad \text{and} \label{alpha'} \\
		\beta'_n & = \beta_1 \beta_n - \beta_{n+1} \label{beta'}
	\end{align}
	for $n>0$, where $K \defeq \frac{1}{\sigma} [(1+\theta \nu_1) (\alpha^{}_0 - \alpha^{}_1) + \sigma + \theta \nu^{}_0 ]$.
\end{lemma}
\begin{proof}
	Recall from \eqref{Wright_distrib} that 
	\begin{equation*}
		\beta_n = \E[\cY_\infty] = \int_0^1 \eta^n \pi(\eta) \dd \eta = \frac{d_n}{d_0},
	\end{equation*}
	where 
	\begin{equation*}
		d_n \defeq \int_0^1 \eta^{n+\theta \nu^{}_1 -1} (1-\eta)^{\theta \nu^{}_0 -1} \ee^{-\sigma \eta} \dd \eta, \quad n \geq 0.
	\end{equation*}
	Clearly, $d_n' = - d_{n+1}$, and so
	\begin{equation*}
		\beta_n' = \frac{-d_{n+1} d_0 + d_n d_1}{d_0^2} = \beta_1 \beta_n - \beta_{n+1}
	\end{equation*}
	in line with \eqref{beta'}.
	We now turn to the $\alpha'_n$. A straightforward application of differentiation rules together with \eqref{beta'} to the expression for $\alpha_n$ in Lemma~\ref{lem:alphabeta} brings us to
	\begin{align*}
		\a'_n &= \frac{n\s^{n-1}}{F_n}\frac{\b^{}_{n+1}-\b^{}_{n+2}}{\b^{}_1-\b^{}_2} + \frac{\s^n}{F_n}\frac{\b^{}_{n+1}-\b^{}_{n+2}}{\b^{}_1-\b^{}_2}\frac{\b^{}_2-\b^{}_3}{\b^{}_1-\b^{}_2}-\frac{\s^n}{F_n}\frac{\b^{}_{n+2}-\b^{}_{n+3}}{\b^{}_1-\b^{}_2}\\
		&= \frac{n}{\s}\a^{}_n + \frac{1+\tt}{\s}\a^{}_1\a^{}_n-\frac{n+1+\tt}{\s}\a^{}_{n+1}.
	\end{align*}
	Applying now \eqref{recursion}, we obtain
	\begin{align*}
		 \a'_n &= \frac{1}{\s} [\s \a^{}_{n-1} - (1+\s+\t)\a^{}_n +(1+\tt)\a^{}_1\a^{}_n ]\\
		 &= \a^{}_{n-1}-K\a^{}_n,
	\end{align*}
	which shows \eqref{alpha'}.
\end{proof}
\begin{proof}[Proof of Theorem~\ref{thm:derivatives}]
	Let us first consider $q^{}_{\cA}(1,0)$ from Corollary~\ref{coro:diff_marg_f_q} and rewrite it as
	\begin{equation*}
		q^{}_{\cA}(1,0) = \theta\nu^{}_0 \frac{A}{A-B} \quad \text{with} \; A \defeq \sum_{n>0} \alpha^{}_{n-1} \beta_n \quad \text{and} \; B \defeq \sum_{n>0} \alpha^{}_{n} \beta_n.
	\end{equation*}
	Differentiating yields
	\begin{equation}\label{muprime}
		q_{\cA}'(1,0) = \theta \nu^{}_0 \frac{B'A - A'B}{(A-B)^2}.
	\end{equation}
	Evaluating $A'$ with the help of Lemma~\ref{lem:alpha'beta'} gives
	\begin{equation*}
		\begin{split}
			A' & = \sum_{n>0} (\alpha'_{n-1} \beta_n + \alpha^{}_{n-1} \beta'_n) = \sum_{n>0} \alpha'_{n} \beta_{n+1} + \sum_{n>0} \alpha^{}_{n-1} \beta'_n \\
			& = \sum_{n>0} (\alpha^{}_{n-1} - K \alpha^{}_n) \beta_{n+1} + \sum_{n>0} \alpha^{}_{n-1} (\beta_1 \beta_n - \beta_{n+1}) \\
			& = - K \sum_{n>0} \alpha^{}_n \beta_{n+1} + \beta_1 \sum_{n>0} \alpha^{}_{n-1} \beta_n = (\beta_1-K) A + K \beta_1,
		\end{split}
	\end{equation*}
	where we have used in the second step that $\alpha^{}_0 = 1$, so $\alpha'_0 = 0$. In a similar way, we obtain
	\begin{equation*}
		\begin{split}
			B' & = \sum_{n>0} (\alpha_{n-1} - K \alpha^{}_n) \beta_n + \sum_{n>0} \alpha^{}_n (\beta_1 \beta_n - \beta_{n+1}) \\
			& = \b_1 + (\b_1-K) \sum_{n>0} \alpha^{}_n \beta_n = (\beta_1 - K) B + \beta_1.
		\end{split}
	\end{equation*}
	So \eqref{muprime} evaluates to
	\begin{equation*}
		q_{\cA}'(1,0) = \frac{\theta \nu^{}_0 \b_1}{(A-B)^2}(A-KB) = \frac{\theta \nu^{}_0\b_1}{(A-B)^2} \sum_{n>0} \a'_n \b_n,
	\end{equation*}
	where we have used in the second step that, by the definitions of $A$ and $B$ as well as Lemma~\ref{lem:alpha'beta'}, we have
	\begin{equation*}
		A - KB = \sum_{n>0} (\a^{}_{n-1} - K \a^{}_n) \b_n= \sum_{n>0} \a'_n \b_n.
	\end{equation*}
	This proves \eqref{mu10'}. To verify \eqref{mu01'}, we use Corollary~\ref{coro:diff_marg_f_q} to write 
	\begin{equation*}
		q_{\cA}(0,1) = \theta\nu^{}_1 \frac{D-E}{D} \quad \text{with } \; D \defeq \sum_{n>0} \alpha^{}_{n-1} \delta_n \quad \text{and} \; E \defeq \sum_{n>0} \alpha^{}_{n} \delta_n,
	\end{equation*}
	where, for $n>0$, $\delta_n \defeq \b_{n-1}-\b_n$. Noting that $\delta'_n = \beta_1 \delta_n - \delta_{n+1}$, one then proceeds in a way perfectly analogous to the proof of \eqref{mu10'}. The signs of $q_{\cA}'(1,0)$ and $q_{\cA}'(0,1)$ now follow immediately since, for all $n>0$, $\a'_n>0$ by Proposition~\ref{prop:tailprobs} and $\delta_n>0$.
\end{proof}

\begin{prop}\label{prop:derivatives_det}
	The mutation rates on the ancestral line in the deterministic limit at stationarity satisfy
	\begin{equation*} 
		q'_{A}(1,0) = u \nu^{}_0 \frac{p'}{(1-p)^2} > 0 \quad \text{and } \; 
		q'_{A}(0,1) = - u \nu^{}_1 p' < 0,
	\end{equation*}
	where 
	\begin{equation*}
		p' = \frac{1}{2 u \nu^{}_1} \Big ( 1 - \frac{u+s-2u \nu^{}_1}{\sqrt{(u+s)^2 - 4 s u \nu^{}_1}} \Big ) > 0.
	\end{equation*}
\end{prop}
\begin{proof}
	The expressions for $p'$, $q'_{A}(1,0)$, and $q'_{A}(0,1)$ follow from \eqref{eq_p} and Corollary~\ref{mutrates_det} via elementary differentiation. Their signs then result from the fact that
	\begin{equation*}
		u+s-2u \nu^{}_1 = \sqrt{(u+s-2u \nu^{}_1)^2} = \sqrt{(u+s)^2 - 4 u \nu^{}_1 (u+s) + 4 (u \nu^{}_1)^2} < \sqrt{(u+s)^2 - 4 s u \nu^{}_1}.
	\end{equation*}
\end{proof}

\subsection{The mutation fluxes}
\begin{figure}
	\psfrag{1e-4}{$\times 10^{-4}$.}
	\includegraphics[width=0.6 \textwidth]{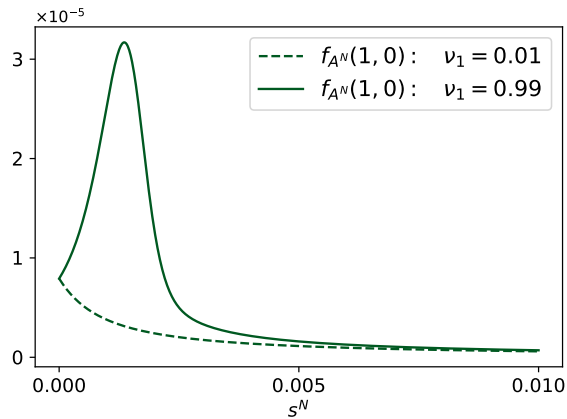}
	\caption{\label{fig:fluxes} Mutation fluxes at stationarity in the finite system as a function of $s^N$ for two values of $\nu_1$. Further parameters: $N=10^4, u^N=8 \cdot 10^{-4}$.}
\end{figure}
The mutation fluxes are important for comparison with data. We concentrate on the finite-$N$ case here; the behaviour in the limits is very similar. Recall that $f_{A^N}(0,1)=f_{A^N}(1,0)$ due to stationarity. Figure~\ref{fig:fluxes} shows how the fluxes vary with $s^N$. In the realistic case of large $\nu_1$ (and hence small $\nu_0$), the mutation flux initially \emph{increases} and then decreases with $s^N$. This may seem astonishing at first sight but has a simple explanation, which reveals itself when comparing with Figures~\ref{fig:anc_dist} and \ref{fig:mutrates}. At $s^N=0$, the common ancestor type process is simply the neutral mutation process at rates $q^{}_{A^N}(0,1)=u^N \nu_1$ and $q^{}_{A^N}(1,0)=u^N \nu_0$, respectively, so the proportion of type-1 individuals on the ancestral line at stationarity is ${\P(A^N_\infty=i)=\nu_i}$ for $i \in \{0,1\}$, and the stationary flux is $u^N \nu_1 \nu_0=u^N (1-\nu_0) \nu_0$ in either direction. Consider now the case $\nu_1=0.99$ (see Section~\ref{sec:biology} for why we chose the parameter values). With increasing $s^N$, $\P(A^N_\infty=0)$ increases; but type-0 individuals have the larger of the two mutation rates, namely $u^N \nu_1$, so the flux initially increases. This behaviour is only reversed at higher values of $s^N$, where type 0 has become dominant and the flux mainly experiences the decrease with $s^N$ of the deleterious mutation rate. In contrast (and unsurprisingly), the flux decreases monotonically with $s^N$ in the unrealistic situation where $\nu_0$ is close to 1.

\section{Theory and observation}
\label{sec:biology}
It is not our intention to provide a close comparison with the observed data --- after all, this is impossible with the simple two-type model. Nevertheless, we can indeed gain some principal insight in the spirit of a proof of concept, in particular by connecting the reduced mutation fluxes mentioned in the introduction to our ancestral structures.

The two-type model may be considered a toy model for a situation with a single wildtype allele or a set of such alleles, which are selectively favoured and lumped into type 0; and a large number of mutant alleles, which are disfavoured and lumped into type 1. The mutation rates between wildtype and mutant are then approximated by $u^N\nu_0 $ and $u^N\nu_1$; mutations within the wildtype and mutant classes are `invisible' in the two-type picture. Let us now discuss our results in the light of the mutation fluxes observed in phylogeny and pedigree studies.

The choice of mutation rates underlying most of our figures is motivated by the mitochondrial control region, to which most of the comparisons of pedigree and phylogenetic mutation rates refer (for example, \citet{Parsons97,Sigurdardottir00,Santos05}; the recent study of mutation rates in pedigrees and phylogeny in complete mitochondrial genomes by \citet{Connell22} reports very similar results). We set $u^N=8 \cdot 10^{-4}$; this is motivated by the overall mutation probability of $\approx 1.6 \cdot 10^{-6}$ per site and generation in the mitochondrial genome \citep{Connell22}, the size of the control region of $\approx 1000$ base pairs, and the guess that roughly half the mutations of the wildtype are neutral, that is, remain within the wildtype class (and are hence invisible in our model). The choice of $\nu_1=0.99$ and $\nu_0=0.01$ in our figures then amounts to assuming that every wildtype allele turns into a deleterious mutant at probability $7.9 \cdot 10^{-4}$ per generation, while every allele in the unfit class is restored to wildtype with probability $8 \cdot 10^{-6}$ per generation. These figures can, of course, be only wild guesses. The inverted choice $\nu_1=0.01$ and $\nu_0=0.99$ is present in our figures for comparison only, not meant to be biologically reasonable in any way.

We are now ready to discuss the mutation fluxes. Indeed, the observations refer to what we call mutation \emph{fluxes} (as opposed to rates), namely the mean number of mutation events per individual and generation, regardless of the parent and offspring's types. The fluxes in the two-type model only account for the deleterious and beneficial mutations; they are blind to the neutral mutations within the wildtype and the mutant classes. This is why Figure~\ref{fig:fluxes} does not tell the whole story and we have to start afresh.

In the pedigree study, one observes all mutations, from any members of the population (or the sample) to any offspring, regardless of whether their lines will be successful in the long run. In phylogeny, one only sees the mutations on the ancestral line, that is, from parents that are ancestors to offspring that have also survived in the long run. But which value of $s^N$ should be considered?

We are, of course, in no position to estimate the selection coefficient or to at least come up with an educated guess; firstly, because this is one of the most difficult tasks in population genetics, and second, because this requires working with a half-way reasonable fitness landscape on a large allelic space, rather than with a two-type model. What we can say, however, is that, for type 0 and 1 to be interpretable as wildtype and mutant, type 0 must be the dominant type by far and make up, say, at least 90\% of the population. Figure~\ref{fig:anc_dist} tells us that, with our mutation parameters, this requires $s^N \gtrapprox 0.008$. For definiteness, we set $b_1^N = 0.9$ and correspondingly $s_* \approx 0.008$, with which we work from now on. By Figure~\ref{fig:anc_dist}, we also have $\P(A^N_\infty=1) \approx 1.3 \cdot 10^{-4}$. Figure~\ref{fig:mutrates} further tells us that $q^N_{A^N}(0,1) \approx 9 \cdot 10^{-7}$ and $q^N_{A^N}(1,0) \approx 0.007$.

In the pedigree setting, it is reasonable to assume that the total mutation rate (that is, neutral plus selective) is the same in both classes. Let $v_0$ and $v_1$ be the neutral mutation rate of type-0 and type-1 individuals, respectively. In agreement with the observed total mutation rate, we therefore arrange $v_0$ and $v_1$ so that $v_0+u^N \nu_1 = 1.6 \cdot 10^{-3} = v_1+u^N \nu_0$, so $v_0 \approx 8.08 \cdot 10^{-4}$ and $v_1\approx 1.592 \cdot 10^{-3}$. The total expected mutation flux in the pedigree setting therefore is
\begin{equation}\label{pedigree_flux}
	\big (1-b_1^N \big ) [v_0+u^N \nu_1] + b_1^N [v_1+ u^N \nu_0] = 1.6 \cdot 10^{-3}.
\end{equation}
In contrast, the total mutation flux in the phylogenetic setting is
\begin{equation}\label{phylo_flux}
	\P(A^N_\infty=0) [v_0 + q^N_{A^N}(0,1)] + \P(A^N_\infty=1) [v_1 + q^N_{A^N}(1,0)] \approx 8.1 \cdot 10^{-4}.
\end{equation}
This reduction by a factor of close to 2 simply comes from the fact that most of the (descendants of the) deleterious mutations are lost in the long run in the phylogenetic setting, due to purifying selection. More generally, neglecting the small terms in \eqref{pedigree_flux} and \eqref{phylo_flux} (see the numerical values above), we see that the phylogenetic reduction factor is close to $(v_0+u^N)/v_0$ and can be arranged to take any value $\geq 1$ by a suitable choice of $v_0$. To explain the observed factor of about 10 solely by purifying selection, one needs $v_0 \approx u^N/9$, which seems unrealistically small. Nevertheless, the results of the toy model make the effect of selection on the mutation fluxes very transparent and confirm that purifying selection may substantially contribute to the observed phenomonon.

\subsection*{Acknowledgment}
This work was funded by the Deutsche Forschungsgemeinschaft (DFG, German Research Foundation) --- Project-ID 317210226 --- SFB 1283. It is our pleasure to thank two referees for extremely thorough reading and insightful comments.

\section{Appendix~A: recursions for the sampling probabilities \(b_n^N\)} \label{sec:app_A}
As previously done for the diffusion and deterministic limits (see \citet{Baake} and \citet{Baake1}), \citet[Section 2.2]{Esercito} have introduced the killed ASG in the finite case, a modification of the ASG that allows to determine whether all individuals in a sample at present are unfit. The crucial observation is that the type of an individual at present is uniquely determined by the most recent mutation that appears on its ancestral line. Hence, when we encounter, going backwards in time, a mutation on a line in the ASG, we do not need to follow it further into the past and we can prune it. Furthermore, if the mutation is beneficial, we know that at least one of the potential ancestors is fit and, because fit lines have the priority at every branching event, the individual at present will also be fit. We therefore kill the process, that is, send it to a cemetary state \(\Delta\). (Note that we cannot say which line is the actual ancestor.) This leads us to the following definition.
\begin{definition}[killed ASG, finite case]
	The killed ASG in the finite case starts with one line emerging from each of the \(n\in [N]\) individuals present in the sample. Each line branches at rate \(s^N(N-n)/N\), every ordered pair of lines coalesces at rate \(1/N\), every line is pruned at rate \(u^N\nu_1\), and the process moves to \(\Delta\) at rate \(u^N\nu_0\) per line.
\end{definition}
Let now \(R^N\defeq (R^N_r)_{r\geq 0}\) be the line-counting process of the killed ASG. It is a continuous-time Markov chain on \(N_{0,\Delta}\defeq [N]_0\cup \Delta\) and with transition rates
\begin{equation*}
	q^{}_{R^N}(n,n+1) = s^Nn\frac{N-n}{N},\hspace{0.9cm}
	q^{}_{R^N}(n,n-1) = n\frac{n-1}{N} + u^N\nu_1 n,\hspace{0.9cm}
	q^{}_{R^N}(n,\Delta) = u^N\nu_0 n
\end{equation*}
for \(n\in[N]\); all other rates are 0.
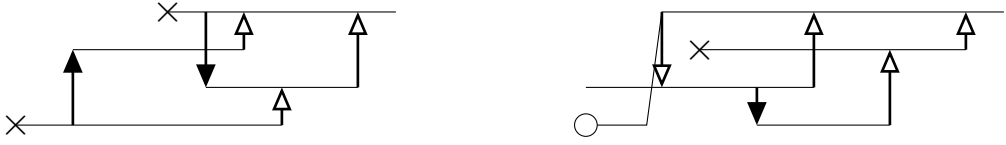
\begin{figure}
	\begin{tikzpicture}
	\draw (0,0)--(3.5,0);
	\draw (2.5,0.5)--(4.5,0.5);
	\draw (0.75,1)--(3,1);
	\draw (2,1.5)--(5,1.5);
	
	\draw[-{open triangle 45},line width=1] (4.5,0.5) -- (4.5,1.5);
	\draw[-{open triangle 45},line width=1] (3.5,0) -- (3.5,0.5);
	\draw[-{open triangle 45},line width=1] (3,1) -- (3,1.5);
	\draw[-{triangle 45},line width=1] (2.5,1.5) -- (2.5,.5);
	\draw[-{triangle 45},line width=1] (0.75,0) -- (0.75,1);
	
	\node at (2,1.5) {\scalebox{1.5}{\(\times\)}};
	\node at (0,0) {\scalebox{1.5}{\(\times\)}};
	
	\draw (7.5,0)--(8.3,0)--(8.5,1.5)--(13,1.5);
	\draw (9,1)--(12.5,1);
	\draw (7.5,0.5)--(10.5,0.5);
	\draw (9.75,0)--(11.5,0);

	\draw[-{open triangle 45},line width=1] (10.5,0.5) -- (10.5,1.5);
	\draw[-{open triangle 45},line width=1] (12.5,1) -- (12.5,1.5);
	\draw[-{open triangle 45},line width=1] (11.5,0) -- (11.5,1);
	\draw[-{triangle 45},line width=1] (9.75,0.5) -- (9.75,0);
	\draw[-{open triangle 45},line width=1] (8.5,1.5)--(8.5,0.5);
	
	\node at (9,1) {\scalebox{1.5}{\(\times\)}};
	\node[draw,circle,inner sep=3pt,fill=white!100] at (7.5,0) {};
	
\end{tikzpicture}
	\caption{The killed ASG either absorbs in the state \(0\) (left) or in the cemetary state \(\Delta\) due to a beneficial mutation (right).\label{Fig:kASG_finite}}
\end{figure}
Clearly, the states \(0\) and \(\Delta\) are absorbing while all other states are transient. Figure \ref{Fig:kASG_finite} shows the two different absorption scenarios. 
The connection with the sampling probabilities is now made by a special case of Corollary 2.4 of \citet{Esercito}.
\begin{coro}[sampling probabilities are absorption probabilities] \label{coro:b_n}
	The sampling probabilities and the line-counting process of the killed ASG are related via
	\begin{equation*}
		b^N_n= \E \Big [ \frac{Y_\infty^ {\underline n}}{N^{\underline n}}\Big ] = \P(R^N \text{ absorbs in } 0 \mid R^N_0=n),\hspace{1cm} n\in [N]_0.
	\end{equation*}
\end{coro}

This connection now yields the sampling recursion.

\begin{coro}[sampling recursion]
	The sampling probabilities satisfy the recursion
	\begin{equation*}
		\Big (\frac{n-1}{N}+ s^N\frac{N-n}{N} + u^N \Big )b^N_n = \Big (\frac{n-1}{N}+u^N\nu_1\Big )b^N_{n-1} + s^N \frac{N-n}{N}b^N_{n+1}, \quad n \in [N],
	\end{equation*}
	together with the boundary conditions \(b^N_0=1\) and \(b^N_{N+1}=0\).
\end{coro}
\begin{proof}
	The proof is a simple first-step decomposition of the absorption probabilities for \(n \in [N]\), which respects the fact that \(\Delta\) is absorbing.
\end{proof}

\section{Appendix~B: The process \(\mathbb{A}^N\)}\label{sec:app_B}
In Section \ref{section_backward}, we have introduced and verbally described the process \(\mathbb{A}^N\defeq(\mathbb{A}^N_t)_{t\in [0,T]}\), defined via \(\mathbb{A}_t^N\defeq (\Xi_t^N,H_{(T-t)-}^N)\). We now consider the process in detail and prove that, at stationarity, it satisfies the Markov property. This is a non-trivial extension of the classical result for Markov chains that states that the time reversal of a Markov chain at stationarity satisfies the Markov property. In our case, we only do a time reversal of the process $H^N$, while $\Xi^N$ continues to run forward in time. The key to the proof will be the fact that our processes are coupled via the graphical representation, where past and future are independent.

In order to simplify the notation, we will omit the upper index \(N\) from processes and parameters in what follows. Let \(\xi \defeq (\xi_k)_{k\in[N]}, \xi'\defeq (\xi'_k)_{k\in[N]}\in\left\{0,1\right\}^N\) be two fixed configurations of \(\Xi\) and let \(\eta \defeq (\eta_k)_{k\in[N]},\eta'\defeq (\eta_k)_{k\in[N]} \in ([N]^*_{\infty})^N\) be two fixed configurations of \(H\); for future need we use the convention that \(i\leq j*\) if and only if \(i \leq j\).
\begin{prop}\label{prop:A_is_Markov}
	Assume that \(\Xi_0\) and \(H_0\) are chosen independently and according to the stationary distributions of the processes \(\Xi\) and \(H\), respectively. Then, for any \(0\leq t\leq T\), the law of \(\mathbb{A}\) is invariant in time and given by the product of the stationary distributions of \(\Xi\) and \(H\). Moreover, the process \(\mathbb{A}\) satisfies the Markov property:
	\begin{align*}
		&\P\big(\Xi_{t+\epsilon}=\xi', H_{(T-(t+\epsilon))-}=\eta' \mid \Xi_{t}=\xi,H_{(T-t)-} =\eta, \overline \Xi=\overline \xi, \overline H = \overline \eta \big)\\
		&\hspace{0.5cm} = \P\big(\Xi_{t+\epsilon}=\xi', H_{(T-(t+\epsilon))-}=\eta' \mid \Xi_{t}=\xi,H_{(T-t)-} =\eta\big),
	\end{align*}
	where \(\overline \Xi \defeq (\Xi_{t_1}, \ldots\Xi_{t_n})\) and \(\overline H\defeq (H_{(T-t_1)-},\ldots,H_{(T-t_n)-})\) with \(0\leq t_1<\ldots<t_n<t\), \(\overline \xi\in\big(\left\{0,1\right\}^N\big)^n\), and \(\overline \eta\in\big(([N]^*_\infty)^N\big)^n\).
\end{prop}
\begin{proof}
	The invariance and the product form of the law of \(\mathbb{A}\) have already been proved in Section \ref{section_backward}. For the Markov property, note first that, thanks to standard properties of the conditional probability, we have
	\begin{align}
		\begin{split}
		&\P\big(\Xi_{t+\epsilon}=\xi', H_{(T-(t+\epsilon))-}=\eta' \mid \Xi_{t}=\xi,H_{(T-t)-} =\eta, \overline \Xi=\overline \xi,\overline H = \overline \eta\big)\\
		&\hspace{0.5cm}=\P\big(\Xi_{t+\epsilon}=\xi'\mid H_{(T-(t+\epsilon))-}=\eta', \Xi_{t}=\xi,H_{(T-t)-} =\eta, \overline \Xi=\overline \xi,\overline H = \overline \eta\big)\cdot\\
		&\hspace{1cm}\cdot\P\big(H_{(T-(t+\epsilon))-} = \eta'\mid \Xi_{t}=\xi,H_{(T-t)-} =\eta, \overline \Xi=\overline \xi,\overline H = \overline \eta\big).
		\end{split}\label{eq:lemma_markov_1}
	\end{align}
	We start by analysing the first factor on the right-hand side.	We know that \(\Xi_{t+\epsilon}\) is a deterministic function, say \(g\), of the value of \(\Xi_t\) and the graphical picture in \((t,t+\epsilon]\). So we have
	\begin{equation}
		\begin{split}
			&\P\big(\Xi_{t+\epsilon}=\xi'\mid H_{(T-(t+\epsilon))-}=\eta', \Xi_{t}=\xi,H_{(T-t)-} =\eta, \overline \Xi=\overline \xi,\overline H = \overline \eta\big)\\ 
			&\hspace{0.5cm} = \P\big(g(\Xi_{t},G((t,t+\epsilon]))=\xi'\mid H_{(T-(t+\epsilon))-}=\eta', \Xi_{t}=\xi,H_{(T-t)-} =\eta, \overline \Xi=\overline \xi,\overline H = \overline \eta\big)\\
			&\hspace{0.5cm} = \P\big(g(\xi,G((t,t+\epsilon]))=\xi'\mid H_{(T-(t+\epsilon))-}=\eta', \Xi_{t}=\xi,H_{(T-t)-} =\eta, \overline \Xi=\overline \xi,\overline H = \overline \eta\big)\\
			&\hspace{0.5cm} = \P\big(g(\xi,G((t,t+\epsilon]))=\xi'\mid H_{(T-(t+\epsilon))-}=\eta', \Xi_{t}=\xi,H_{(T-t)-} =\eta,\overline H = \overline \eta\big),
		\end{split}\label{eq:lemma_markov_2}
	\end{equation}
	where, in the last step, we have used that \(G\big((t,t+\epsilon]\big)\) is independent of \(\overline \Xi\). We also know that \(\overline H\) is a deterministic function, say \(h\), of \(H_{(T-t)-}\) and the graphical picture in \([0,t)\), so we have
	\begin{equation}
		\begin{split}
			&\P\big(g(\xi,G((t,t+\epsilon]))=\xi'\mid H_{(T-(t+\epsilon))-}=\eta', \Xi_{t}=\xi,H_{(T-t)-} =\eta,\overline H = \overline \eta\big)\\
			&\hspace{0.5cm} = \P\big(g(\xi,G((t,t+\epsilon]))=\xi'\mid H_{(T-(t+\epsilon))-}=\eta', \Xi_{t}=\xi,H_{(T-t)-} =\eta,\\
			&\hspace{1.5cm}h(H_{(T-t)-},G([0,t) )= \overline \eta\big)\\
			&\hspace{0.5cm} = \P\big(g(\xi,G((t,t+\epsilon]))=\xi'\mid H_{(T-(t+\epsilon))-}=\eta', \Xi_{t}=\xi,H_{(T-t)-} =\eta,\\
			&\hspace{1.5cm} h(\eta,G([0,t) )= \overline \eta\big)\\
			&\hspace{0.5cm} = \P\big(\Xi_{t+\epsilon}=\xi'\mid H_{(T-(t+\epsilon))-}=\eta', \Xi_{t}=\xi,H_{(T-t)-} =\eta\big),
		\end{split}\label{eq:lemma_markov_3}
	\end{equation}
	where, in the last step, we have used the independence of \(G\big([0,t)\big)\) and \(G\big((t,t+\epsilon]\big)\).
	
	We now consider the second factor on the right-hand side of \eqref{eq:lemma_markov_1}. Since \(H_{(T-(t+\epsilon))-}\) is independent of both \(\Xi_t\) and \(\overline \Xi\), we have
	\begin{align}
		\begin{split}
			&\P\big(H_{(T-(t+\epsilon))-} = \eta'\mid \Xi_{t}=\xi,H_{(T-t)-} =\eta, \overline \Xi=\overline \xi,\overline H = \overline \eta\big)\\
			&\hspace{0.5cm} = \P\big(H_{(T-(t+\epsilon))-} = \eta'\mid H_{(T-t)-} =\eta,\overline H = \overline \eta\big)\\
			&\hspace{0.5cm} = \P\big(H_{(T-(t+\epsilon))-} = \eta'\mid H_{(T-t)-} =\eta\big),
		\end{split}\label{eq:lemma_markov_4}
	\end{align}
	where, in the last step, we have used that the time reversal of \(H\) satisfies the Markov property at stationarity.
	
	We now use \eqref{eq:lemma_markov_2}--\eqref{eq:lemma_markov_4} to rewrite \eqref{eq:lemma_markov_1} as
	\begin{equation}
		\begin{split}
			&\P\big(\Xi_{t+\epsilon}=\xi', H_{(T-(t+\epsilon))-}=\eta' \mid \Xi_{t}=\xi,H_{(T-t)-} =\eta, \overline \Xi=\overline \xi,\overline H = \overline \eta\big)\\
			&\hspace{0.5cm}=\P\big(\Xi_{t+\epsilon}=\xi'\mid H_{(T-(t+\epsilon))-}=\eta', \Xi_{t}=\xi,H_{(T-t)-} =\eta\big)\cdot\\
			&\hspace{1.5cm}\cdot\P\big(H_{(T-(t+\epsilon))-} = \eta'\mid H_{(T-t)-} =\eta\big).
		\end{split}\label{eq:lemma_markov_5}
	\end{equation}
	In addition, we have
	\begin{equation}
		\begin{split}
			&\P\big(\Xi_{t+\epsilon}=\xi', H_{(T-(t+\epsilon))-}=\eta' \mid \Xi_{t}=\xi,H_{(T-t)-} =\eta\big)\\
			&\hspace{0.5cm}=\P\big(\Xi_{t+\epsilon}=\xi'\mid H_{(T-(t+\epsilon))-}=\eta', \Xi_{t}=\xi,H_{(T-t)-} =\eta\big)\cdot\\
			&\hspace{1.5cm}\cdot\P\big(H_{(T-(t+\epsilon))-}=\eta'\mid \Xi_{t}=\xi,H_{(T-t)-} =\eta\big)\\
			&\hspace{0.5cm}=\P\big(\Xi_{t+\epsilon}=\xi'\mid H_{(T-(t+\epsilon))-}=\eta', \Xi_{t}=\xi,H_{(T-t)-} =\eta\big)\cdot\\
			&\hspace{1.5cm}\cdot\P\big(H_{(T-(t+\epsilon))-}=\eta'\mid H_{(T-t)-} =\eta),
		\end{split}\label{eq:lemma_markov_6}
	\end{equation}
	where, in the last step, we have used the independence of \(H_{(T-(t+\epsilon))-}\) and \(\Xi_t\). From \eqref{eq:lemma_markov_5} and \eqref{eq:lemma_markov_6} we can conclude that the process \(\mathbb{A}\) has the Markov property.
\end{proof}
Computing the rates of the process \(\mathbb{A}\) is an interesting exercise that yields additional insight; however, it is somewhat lengthy, why we spare it to the reader here. It will appear in the forthcoming PhD thesis of Enrico Di Gaspero.


\begin{thebibliography}{10}

	\bibitem[Adams et al.(2018)]{Adams_etal_18}
	R.H. Adams, D.R. Schield, D.C. Card, and T.A. Castoe. \textit{Assessing the impacts of positive selection on coalescent-based species tree estimation and species delimitation}. Syst. Biol. \textbf{67} (2018), 1067--1090.

	\bibitem[Baake et al.(2018)]{Baake1}
	E. Baake, F. Cordero and S. Hummel. \textit{A probabilistic view on the deterministic mutation-selection equation: dynamics, equilibria, and ancestry via individual lines of descent}. J. Math. Biol. \textbf{77} (2018), 795--820. 

	\bibitem[Baake et al.(2023)]{Esercito}
	E. Baake, L. Esercito, and S. Hummel. \textit{Lines of descent in a Moran model with frequency-dependent selection and mutation}, Stoch. Processes Appl. \textbf{160} (2023), 409--457.
	
	\bibitem[Baake and Wakolbinger(2018)]{Baake}
	E. Baake and A. Wakolbinger. \textit{Lines of descent under selection}. J. Stat. Phys. \textbf{172} (2018), 156--174. 

	\bibitem[Boenkost et al.(2021)]{Boenkost21a}
	F. Boenkost, A. Gonzalez Casanova, C. Pokalyuk, and A.Wakolbinger. \textit{Haldane's formula in {C}annings models: the case of moderately weak selection}. Electron. J. Probab. \textbf{26} (2021), article no. 4, 36 pp.
	
	\bibitem[Boenkost et al.(2022)]{Boenkost21b}
	F. Boenkost, A. Gonzalez Casanova, C. Pokalyuk, and A.Wakolbinger. \textit{Haldane's formula in {C}annings models: the case of moderately strong selection}. J. Math. Biol. \textbf{83} (2021), paper no. 70, 31 pp.
	
	\bibitem[Connell et al.(2022)]{Connell22}
	J.R. Connell, M.C. Benton, R.A. Lea et al. (5 coauthors). \textit{Pedigree derived mutation rate across the entire mitochondrial genome of the Norfolk island population.} Sci. Rep. \textbf{12} (2022), paper no. 6827, 11 pp.
	
	\bibitem[Cordero(2017a)]{Cordero}
	F. Cordero. \textit{Common ancestor type distribution: a Moran model and its deterministic limit}. Stoch. Processes Appl. \textbf{127} (2017), 590--621.
	
	\bibitem[Cordero(2017b)]{Cordero2}
	F. Cordero. \textit{The deterministic limit of the Moran model: a uniform central limit theorem}. Markov Processes Relat. Fields \textbf{23} (2017), 313--324.
	
	\bibitem[Donnelly and Kurtz (1999)]{don}
	P. Donnelly and T.G. Kurtz. \textit{Particle representations for measure-valued population models}. Ann. Probab. \textbf{27} (1999), 166-205.
	
	\bibitem[Durrett(2008)]{Durrett}
	R. Durrett, \textit{Probability Models for DNA Sequence Evolution}, 2nd ed., Springer, New York (2008).
	
	\bibitem[Etheridge(2000)]{Etheridge00}
	A. Etheridge, \textit{An Introduction to Superprocesses}, Amer. Math. Soc., Providence, RI (2011).

	\bibitem[Etheridge(2011)]{Etheridge11}
	A. Etheridge, \textit{Some Mathematical Models from Population Genetics}, Springer, Berlin (2011).
	
	\bibitem[Ewens and Grant(2005)]{Ewens_Grant_05}
	W.J. Ewens and G.J. Grant, \textit{Statistical Methods in Bioinformatics}, 2nd ed., Springer, New York (2005).
	
	\bibitem[Ethier and Kurtz(1986)]{kurtz}
	S.N. Ethier and T.G. Kurtz. \textit{Markov Processes: Characterization and Convergence}. Wiley, New York (1986). reprint 2005
	
	\bibitem[Fearnhead(2002)]{Fearnhead}
	P. Fearnhead. \textit{The common ancestor at a nonneutral locus}. J. Appl. Prob. \textbf{39} (2002), 38--54.
	
	\bibitem[Georgii and Baake(2003)]{Georgii}
	H.O. Georgii and E. Baake. \textit{Supercritical multitype branching process: the ancestral types of typical individuals}. Adv. Appl. Prob. \textbf{35} (2003), 1090-1110.
	
	\bibitem[Kluth et al.(2013)]{Sandra}
	S. Kluth, T. Hustedt and E. Baake. \textit{The common ancestor process revisited}. Bull. Math. Biol.\textbf{75} (2013), 2003-2027.
	
	\bibitem[Krone and Neuhauser(1997)]{ASG}
	S.M. Krone and C. Neuhauser. \textit{Ancestral processes with selection}. Theor. Popul. Biol. \textbf{51} (1997), 210--237.
	
	\bibitem[Lenz et al.(2015)]{len}
	U. Lenz, S. Kluth, E. Baake and A. Wakolbinger. \textit{Looking down in the ancestral selection graph: A probabilistic approach to the common ancestor type distribution}. Theor. Popul. Biol. \textbf{103} (2015), pp. 27--37.
	
	\bibitem[Parsons et al.(1997)]{Parsons97}
	T.J.~Parsons, D.S.~Muniec, K.~Sullivan et al. (11 co-authors), \textit{A high observed substitution rate in the human mitochondrial DNA control region}, Nat. Genet. \textbf{15} (1997), 363--368.

	\bibitem[Santos et al.(2005)]{Santos05}
	C.~Santos, R.~Montiel, B.~Sierra et al. (6 coauthors), \textit{Understanding differences between phylogenetic and pedigree-derived mtDNA mutation rate: a model using families from the Azores Islands (Portugal)}. Mol. Biol. Evol. \textbf{22} (2005), 1490--1505.

	\bibitem[Sigur{\dh}ard\'{o}ttir et al.(2000)]{Sigurdardottir00}
	S.~Sigur{\dh}ard\'{o}ttir, A.~Helgason, J.R.~Gulcher, K.~Stefansson, and P.~Donnelly, \textit{The mutation rate in the human mtDNA control region}, Am. J. Hum. Genet. \textbf{66} (2000), 1599--1609. 	

	\bibitem[Taylor (2007)]{Taylor07}
	J.~E.~Taylor, \textit{The common ancestor process for a Wright--Fisher diffusion}, Electron. J. Probab. \textbf{12} (2007), paper no. 28, pp.~808--847.
	
	\bibitem[Wakeley(2009)]{Wakeley}
	J.~Wakeley, \textit{Coalescent Theory: An Introduction}. Greenwood Village, Cororado (2009).
	
	\bibitem[Wascher and Kubatko(2022)]{Wascher_Kubatko_22}
	M.~Wascher and L.M.~Kubatko, \textit{On the effects of selection and mutation on species tree inference}, Mol. Phylogenet. Evol. \textbf{179} (2022), 107605.

\end{thebibliography}
\end{document}